\pdfoutput=1
\newif\ifCmp\Cmptrue% Compact exposition?
\newcount\FmtChoice% 0=OMS (outdated), 1=AMS (arXiv), 2=JNSAO (final)
\ifcase\FmtChoice
\documentclass{interact}
\or
\documentclass[a4paper,ngerman,american,reqno]{amsart}
\usepackage{a4wide}
\or
\documentclass[english]{jnsao}
\fi

\usepackage{xcolor}
\usepackage{mathtools}
\usepackage{bm}
\usepackage{tikz}
\usetikzlibrary{matrix}
\usetikzlibrary{arrows.meta}
\usepackage{ao-math-std}
\usepackage{ao-math-symbols}
\usepackage[nameinlink,capitalize]{cleveref}

\usepackage[numbers,sort&compress]{natbib}
\bibpunct[, ]{[}{]}{,}{n}{,}{,}

\ifnum\FmtChoice<2
\newtheorem{theorem}{Theorem}
\newtheorem{lemma}[theorem]{Lemma}

\newtheorem{corollary}[theorem]{Corollary}
\theoremstyle{definition}
\newtheorem{definition}[theorem]{Definition}
\newtheorem{example}[theorem]{Example}
\theoremstyle{remark}
\newtheorem{remark}[theorem]{Remark}
\fi

\DeclareMathOperator\sign{sign}
\DeclareMathOperator\rank{rank}
\DeclareMathOperator\diag{diag}

\ifcase0
\newcommand\minst[2][]{\smash[b]{\min_{#1}} \ #2 \qstq}
\or
\newcommand\minst[2][]{\min_{#1} \quad &#2 \\ \stq}
\fi

\newcommand\compl{\ensuremath\perp}
\newcommand\R{\mathbb R}
\newcommand\setA{\mathcal A}
\newcommand\setD{\mathcal D}
\newcommand\setF{\mathcal F}

\newcommand\setP{\mathcal P}
\newcommand\setS{\mathcal S}

\newcommand\setU{\mathcal U}
\newcommand\setV{\mathcal V}
\newcommand\setUV{\setU_+ \cup \setV_+}
\newcommand\setDt{\setD^t}
\newcommand\setUt{\setU^t}
\newcommand\setVt{\setV^t}
\newcommand\PUt[1]{P_{\setUt_#1}}
\newcommand\PVt[1]{P_{\setVt_#1}}
\newcommand\PU{P_{\setU_+}}
\newcommand\PV{P_{\setV_+}}
\newcommand\PUV{P_{\setUV}}

\newcommand\dx{d_x}

\newcommand\du{d_u}
\newcommand\dv{d_v}

\newcommand\muu{\mu_u}
\newcommand\muv{\mu_v}
\newcommand\Lagr{\mathcal L}
\newcommand\Lc{\Lagr_\compl}
\newcommand\setE{\mathcal E}
\newcommand\setI{\mathcal I}
\newcommand\setZ{\mathcal Z}

\newcommand\Fabs{\setF\tsb{abs}}
\newcommand\Feabs{\setF\tsb{e-abs}}

\newcommand\Fmpcc{\setF\tsb{mpcc}}
\newcommand\Fempcc{\setF\tsb{e-mpcc}}

\newcommand\Jmpcc{J\tsb{mpcc}}
\newcommand\Umpcc{U\tsb{mpcc}}
\newcommand\Hmpcc{H\tsb{mpcc}}

\newcommand\Jempcc{J\tsb{e-mpcc}}

\newcommand\Jabs{J\tsb{abs}}
\newcommand\Uabs{U\tsb{abs}}
\newcommand\Habs{H\tsb{abs}}
\newcommand\Ueabs{\_U\tsb{e-abs}}
\newcommand\Heabs{\_H\tsb{e-abs}}

\newcommand\Jalp{J_{\alpha}}
\newcommand\JE{J_\setE}
\newcommand\JA{J_\setA}
\newcommand\bJalp{\_J_\alpha}
\newcommand\bJE{\_J_\setE}

\newcommand\Jeabs{J\tsb{e-abs}}

\newcommand\alpt{\alpha^t}
\newcommand\alpw{\alpha^w}
\newcommand\sigt{\sigma^t}
\newcommand\sigw{\sigma^w}
\newcommand\zt{z^t}
\newcommand\zw{z^w}

\newcommand\ut{u^t}
\newcommand\uw{u^w}

\newcommand\vt{v^t}
\newcommand\vw{v^w}

\newcommand{\Dom}{D}
\newcommand\Domt{\Dom^t}
\newcommand\Domx{\Dom^x}
\newcommand{\Domzt}{\Dom^{\abs{\zt}}}

\newcommand\Domz{\Dom^{\abs{z}}}

\newcommand\Domtzt{\Dom^{t, \abs{\zt}}}
\newcommand\Domxz{\Dom^{x, \abs{z}}}

\newcommand{\Cspace}{C}
\newcommand{\Cd}{\Cspace^d}
\newcommand{\Cabs}{\Cd_{\text{abs}}}
\newcommand{\Cone}{\Cspace^1}

\newcommand\cA{c_\setA}
\newcommand\cE{c_\setE}
\newcommand\cI{c_\setI}
\newcommand\cZ{c_\setZ}
\newcommand\bcE{\_c_\setE}
\newcommand\bcZ{\_c_\setZ}

\newcommand\calp[1][]{\card{\alpha#1}}
\newcommand\calpt{\card{\alpt}}
\newcommand\calpw{\card{\alpw}}
\newcommand\csig[1][]{\card{\sigma#1}}

\newcommand\lamA{\lambda_\setA}
\newcommand\lamE{\lambda_\setE}
\newcommand\lamI{\lambda_\setI}
\newcommand\lamZ{\lambda_\setZ}
\newcommand\blamE{\_\lambda_\setE}
\newcommand\blamZ{\_\lambda_\setZ}
\newcommand\lamZw{\lamZ^w}

\newcommand\pd{\partial}

\newenvironment{defarray}[2][l]
{\left\{\,#2\,\left|\def\arraystretch{\defarraystretch}\begin{array}{#1}}
{\end{array}\right.\!\right\}}

\ifnum\FmtChoice=1
  
\fi

\newcommand\abstracttext{%
  This work is part of an ongoing effort of comparing
  non-smooth optimization problems in abs-normal form to MPCCs.
  We study the general abs-normal NLP with equality and inequality constraints
  in relation to an equivalent MPCC reformulation.
  We show that kink qualifications and MPCC constraint qualifications of
  linear independence type and Mangasarian-Fromovitz type are equivalent.
  Then we consider strong stationarity concepts with first and second
  order optimality conditions, which again turn out to be equivalent
  for the two problem classes.
  Throughout we also consider specific slack reformulations suggested in
  \cite{Hegerhorst_Steinbach:2019}, which preserve constraint qualifications
  of linear independence type but \emph{not} of Mangasarian-Fromovitz type.
}
\newcommand\keywordlist{%
  Non-smooth NLPs,
  abs-normal form,
  MPCCs,
  linear independence and Mangasarian-Fromovitz type constraint qualifications,
  optimality conditions
}
\newcommand\amscodelist{%
  90C30, % Nonlinear programming
  90C33, % Complementarity and equilibrium problems and variational inequalities (finite dimensions)
  90C46%  Optimality conditions, duality
}
\newcommand\addrIfAM{Leibniz Universit\"at Hannover,
  Institut f\"ur Angewandte Mathematik,
  Wel\-fen\-gar\-ten~1, 30167 Hannover, Germany}
\newcommand\addrIMO{Technische Universit\"at Carolo-Wilhelmina zu Braunschweig,
  Institut f\"ur Mathe\-ma\-ti\-sche Optimierung, Universit\"ats\-platz~2,
  38106 Braunschweig, Germany}

\ifcase\FmtChoice\or\or% JNSAO

\title{Relations Between Abs-Normal NLPs and MPCCs.
  Part 1: Strong Constraint Qualifications}

\date{\ISOToday}

\author{L.C. Hegerhorst-Schultchen%
  \thanks{\addrIfAM. \email{hegerhorst@ifam.uni-hannover.de}}
  \and
  C. Kirches\thanks{\addrIMO. \email{c.kirches@tu-bs.de}}
  \and
  M.C.~Steinbach\thanks{\addrIfAM. \email{mcs@ifam.uni-hannover.de}}
}
\shortauthor{Hegerhorst-Schultchen, Kirches, Steinbach}

\acknowledgements{C.~Kirches discloses financial support by the German Federal Ministry of Education and Research through grants
	n\textsuperscript{o} 05M17MBA-MoPhaPro, 05M18MBA-MOReNet, 01/S17089C-ODINE, and 05M20MBA-LEOPLAN,
	and by Deutsche Forschungsgemeinschaft (DFG) through Priority Programme 1962, grant Ki1839/1-2.}

\manuscriptcopyright{{\copyright} the authors}
\manuscriptlicense{CC-BY-SA 4.0}

\manuscriptsubmitted{2020-07-30}
\manuscriptaccepted{2021-02-12}
\manuscriptvolume{2}
\manuscriptnumber{6672}
\manuscriptyear{2021}
\manuscriptdoi{10.46298/jnsao-2021-6672}

\fi

\begin{document}

\ifcase\FmtChoice% OMS
\author{
  \name{L.C.~Hegerhorst-Schultchen\textsuperscript{a}$^\ast$%
    \thanks{Corresponding author. E-Mail {\tt hegerhorst@ifam.uni-hannover.de}.}
    \and
    C. Kirches\textsuperscript{b}
    \and
    M.C.~Steinbach\textsuperscript{a}}
  \affil{\textsuperscript{a}\addrIfAM.}
  \affil{\textsuperscript{b}\addrIMO.}
}
\title{Relations Between Abs-Normal NLPs and MPCCs.\\
  Part 1: Strong Constraint Qualifications}
\date{\today}
\maketitle
\begin{abstract}
  \abstracttext
\end{abstract}
\begin{amscode}
  \amscodelist
\end{amscode}
\or% AMS
\author{L.C.~Hegerhorst-Schultchen}
\address{L.C. Hegerhorst-Schultchen, \addrIfAM.}
\email{hegerhorst@ifam.uni-hannover.de}
\author{C. Kirches}
\address{C. Kirches, \addrIMO.}
\email{c.kirches@tu-bs.de}
\author{M.C.~Steinbach}
\address{M.C.~Steinbach, \addrIfAM.}
\email{mcs@ifam.uni-hannover.de}
\title[Relations Between Abs-Normal NLPs and MPCCs. Part 1: Strong CQs]
{Relations Between Abs-Normal NLPs and MPCCs\\
  Part 1: Strong Constraint Qualifications}
\begin{abstract}
  \abstracttext
\end{abstract}
\subjclass{\amscodelist}
\maketitle
 \or% JNSAO
\maketitle
\begin{abstract}
  \abstracttext
\end{abstract}

%No keywords and AMS codes :(
\fi

%%%%%%%%%%%%%%%%%%%%%%%%%%%%%%%%%%%%%%%%%%%%%%%%%%%%%%%%%%%%%%%%%%%%%%%%%%%%%%%%%%%%%%%%%%
\section{Introduction}
\label{sec:introduction}
%%%%%%%%%%%%%%%%%%%%%%%%%%%%%%%%%%%%%%%%%%%%%%%%%%%%%%%%%%%%%%%%%%%%%%%%%%%%%%%%%%%%%%%%%%

Nonsmoothness arises frequently
in practical optimization problems from various areas.
In finite dimensions, certain nonsmooth functions
like $\ell_1$ and $\ell_\oo$ norms
can be avoided by smooth remodeling,
but models that involve (possibly nested) absolute value,
maximum, and minimum functions,
models with piecewise definitions and
models with equilibrium conditions or complementarity conditions
lead to genuine nonsmoothness.
An important sub-class of such nonsmooth finite-dimensional optimization problems
is essentially characterized as ``NLPs with finitely many kinks'',
which gives rise to more general standard problem classes like
MPCCs \cite{Luo_Pang_Ralph_1996}
and optimization problems in abs-normal form \cite{Griewank_2013}.

The latter are non-smooth nonlinear optimization problems of the form
 \begin{align*}
   \minst[x]{f(x)}
   g(x) &= 0, \tag{NLP} \label{eq:nlp} \\
   h(x) &\ge 0,
 \end{align*}
where $\Domx\subseteq \R^n$ is open, the objective $f\in\Cd(\Domx,\R)$ is a smooth function
($d \ge 1$)
and the equality and inequality constraints $g\in\Cabs(\Domx,\R^{m_1})$ and
$h\in\Cabs(\Domx,\R^{m_2})$ are non-smooth functions with the non-smoothness exposed in
abs-normal form \cite{Griewank_2013}.
Thus, there exist functions
$\cE\in\Cd(\Domxz,\R^{m_1})$,
$\cI\in\Cd(\Domxz,\R^{m_2})$ and
$\cZ\in\Cd(\Domxz,\R^s)$ with $\Domxz=\Domx\x\Domz$,
$\Domz\subseteq\R^s$ open and \emph{symmetric}
(i.e., $z \in \Domz$ implies $\Sigma z \in \Domz$
for every signature matrix $\Sigma$, see \cref{def:signature}) such that
\begin{align*}
g(x)&=\cE(x,\abs{z}),\\
h(x)&=\cI(x,\abs{z}),\tag{ANF}\label{eq:anf}\\
z&=\cZ(x,\abs{z}) \qtext{with $\pd_2 \cZ(x,\abs{z})$ strictly lower triangular}.
\end{align*}
Note that we introduce one joint \emph{switching constraint} $\cZ$
for $g$ and $h$ and reuse \emph{switching variables} $z_i$
if the same argument occurs inside an absolute value in $g$ and $h$.
Here, components of $z$ can be computed one by one from $x$ and $z_i$, $i<j$,
since $\pd_2 \cZ(x,\abs{z})$ is strictly lower triangular.
In the following we write $z(x)$ to denote this dependence explicitly.
However, $z$ is implicitly defined by $z=\cZ(x,\abs{z})$.
To consider solvability of this system, we use the reformulation $\abs{z_i} = \sign(z_i) z_i$.

\begin{definition}[Signature of $z$]
  \label{def:signature}
  Let $x \in \Domx$. We define the \emph{signature} $\sigma(x)$
  and the associated \emph{signature matrix} $\Sigma(x)$ as
  \begin{align*}
    \sigma(x) &\define \sign(z(x)) \in \set{-1,0,1}^s, &
    \Sigma(x) &\define \diag(\sigma(x)).
  \end{align*}
  A signature vector $\sigma(x) \in \set{-1,1}^s$ is called \emph{definite},
  otherwise \emph{indefinite}.
\end{definition}

With \cref{def:signature},
we can write $\abs{z(x)}=\Sigma(x) z(x)$ and consider the system $z=\cZ(x,\Sigma z)$, which for any \textbf{fixed signature} $\Sigma = \Sigma(\^x)$ becomes a differentiable system.
Then, application of the implicit function theorem yields the existence of a locally unique solution $z(x)$ with Jacobian
\begin{equation*}
 \pd_x z(\^x) = [I - \pd_2 \cZ(\^x, \abs{z(\^x)}) \Sigma]^{-1}
 \pd_1 \cZ(\^x, \abs{z(\^x)}) \in \R^{s \x n}.
\end{equation*}

\begin{definition}[Active Switching Set]
  We call the switching variable $z_i$ \emph{active} if $z_i(x) = 0$.
  The \emph{active switching set} $\alpha$
  consists of all indices of active switching variables, i.e.
  \begin{equation*}
    \alpha(x) \define \defset{i \in \set{1, \dots, s}}{z_i(x) = 0}.
  \end{equation*}
  We denote the number of active switching variables by $\calp[(x)]$
  and the number of inactive ones by $\csig[(x)] \define s - \calp[(x)]$.
\end{definition}

Unconstrained optimization problems in abs-normal form
have recently been introduced by \cite{Griewank_Walther_2016},
and have been shown to come with favorable theoretical properties
that lead to globally convergent solution methods
based on piecewise linearizations
that can be generated by algorithmic differentiation \cite{Griewank_Walther_2018, Griewank_Walther_2019}.
When nonsmooth equality and inequality constraints are added,
the resulting abs-normal NLPs can be compared to MPCCs,
for which a well-established theoretical framework has been developed over the past decades \cite{Luo_Pang_Ralph_1996}.
Comparing the two classes of nonsmooth optimization problems,
the authors reported in \cite{Hegerhorst_et_al:2019:MPEC1}
that they basically form the same problem class,
although with a different representation of nonsmoothness
and not necessarily with the same regularity properties.

\subsection*{Literature}

The abs-normal NLPs considered here are a direct generalization of
unconstrained abs-normal problems developed by Griewank and Walther
\cite{Griewank_2013,Griewank_Walther_2016}.
These problems offer particularly attractive theoretical features
when generalizing KKT theory and stationarity concepts,
and they are tractable by sophisticated algorithms with guaranteed convergence
based on piecewise linearizations
and using algorithmic differentiation techniques \cite{Griewank_Walther_2018,Griewank_Walther_2019}.

Another important class of nonsmooth optimization problems are Mathematical Programs with Complementarity (or Equilibrium) Constraints (MPCCs, MPECs); an overview can be found in the book \cite{Luo_Pang_Ralph_1996}.
Since standard theory for smooth optimization problems cannot be applied, new constraint qualifications and corresponding optimality conditions were introduced.
By now, there is a large body of literature on MPCCs, and we refer to \cite{Ye:2005} for an overview of the basic concepts and theory.
In this paper, constraint qualifications for MPCCs in the sense of linear independence and Mangasarian Fromowitz are considered.
Further, the corresponding stationarity concept (S-stationarity) as well as first and second order optimality conditions are studied.
Details can be found in \cite{Scheel_Scholtes_2000}, \cite{Luo_Pang_Ralph_1996} and \cite{FlegelDiss}.

In \cite{Hegerhorst_et_al:2019:MPEC1} we have shown that unconstrained
optimization problems in abs-normal form are a subclass of MPCCs
and we have studied regularity concepts of linear independence type,
Mangasarian-Fromovitz type and Abadie type.
We have also shown that abs-normal NLPs with general constraints
are equivalent to the class of MPCCs.
In \cite{Hegerhorst_Steinbach:2019} we have generalized
optimality conditions of unconstrained abs-normal problems
to the case with equality and inequality constraints
under the linear independence kink qualification.
More details and additional information about these results as well as about the results in this paper can be found in \cite{Hegerhorst-Schultchen2020}.

\subsection*{Contributions}

We develop a deeper understanding of the commonalities
of these problem classes of NLPs in abs-normal form on the one hand, and MPCCs on the other. We provide a detailed comparative study
of general abs-normal NLPs and MPCCs,
considering constraint qualifications
of linear independence type and Mangasarian Fromovitz type
for the standard formulation
and for a reformulation with absolute value slacks
that was suggested in \cite{Hegerhorst_Steinbach:2019}.
In particular, we show that corresponding constraint qualifications
of abs-normal NLPs and MPCCs are equivalent
and that the linear independence type constraint qualifications
are preserved by the slack reformulation
while this is not the case for a Mangasarian-Fromovitz type constraint qualification.
We then compare optimality conditions of first and second order
for abs-normal NLPs and MPCCs
under the respective linear independence type constraint qualifications.
We show equivalence of the respective first order necessary conditions, kink stationarity and strong stationarity.
We also show how second order conditions for MPCCs can be carried over to abs-normal NLPs. Under suitable additional assumptions, we prove equivalence
of positive (semi-)definiteness of the associated reduced Hessians,
which gives correspondences of second order necessary and sufficient conditions.

We expect that our theoretical results will contribute
to the understanding and further development
of rigorous solution algorithms for abs-normal NLPs. We also expect the results to
facilitate a possible transferral of active-signature methods for abs-normal forms, such as SALMIN \cite{Griewank_Walther_2018}, to MPCCs.

\subsection*{Structure} The remainder of this article is structured as follows.
In \cref{sec:anf-formulations}
we present the general abs-normal NLP and its slack reformulation,
and we formulate the associated kink qualifications and compare them.
In \cref{sec:mpcc} we introduce counterpart MPCCs
for the two formulations of abs-normal NLPs
and compare the associated MPCC-constraint qualifications.
Then, we show equivalence of the regularity concepts
for abs-normal NLPs and MPCCs in \cref{sec:regularity}.
Finally, in \cref{sec:opt-cond} we state optimality conditions
of first and second-order for abs-normal NLPs and MPCCs
and prove equivalence relations between them.
We conclude in \cref{sec:conclusion} and give a brief outlook.

%%%%%%%%%%%%%%%%%%%%%%%%%%%%%%%%%%%%%%%%%%%%%%%%%%%%%%%%%%%%%%%%%%%%%%%%%%%%%%%%%%%%%%%%%%%
\section{Abs-Normal NLPs}
\label{sec:anf-formulations}
%%%%%%%%%%%%%%%%%%%%%%%%%%%%%%%%%%%%%%%%%%%%%%%%%%%%%%%%%%%%%%%%%%%%%%%%%%%%%%%%%%%%%%%%%%%

In this section we consider two formulations for non-smooth NLPs in abs-normal form
that differ in the treatment of inequality constraints.

% =============================================================================
\subsection{General Abs-Normal NLPs}
\label{sec:anf-inequalities}
% =============================================================================

In this paragraph we consider abs-normal NLPs with equalities and inequalities,
obtained by substituting the constraints representation
in abs-normal form \eqref{eq:anf}
into the general non-smooth problem \eqref{eq:nlp}.
Note that we use the variables $(t,\zt)$ instead of $(x,z)$.

\begin{definition}[Abs-Normal NLP]
   Let $\Domt$ be an open subset of $\R^{n_t}$.
   A non-smooth NLP is called an \emph{abs-normal NLP}
   if functions $f \in \Cd(\Domt,\R)$, $\cE \in \Cd(\Domtzt,\R^{m_1})$,
   $\cI \in \Cd(\Domtzt,\R^{m_2})$, and $\cZ \in \Cd(\Domtzt,\R^{s_t})$ with $d \ge 1$ exist
   such that the NLP reads
   \begin{align*}
     \minst[t, \zt]{f(t)}
     & \cE(t, \abs{\zt}) = 0, \\
     & \cI(t, \abs{\zt}) \ge 0, \tag{I-NLP} \label{eq:i-anf} \\
     & \cZ(t, \abs{\zt}) - \zt = 0,
   \end{align*}
   where $\Domzt$ is open and symmetric and
   $\pd_2 \cZ(x,\abs{\zt})$ is strictly lower triangular.
\end{definition}
   The feasible set of \eqref{eq:i-anf} is denoted by
   \begin{align*}
     \Fabs &\define
     \begin{defarray}{(t, \zt)}
       \cE(t, \abs\zt) = 0,\ \cI(t, \abs\zt) \ge 0, \\
       \cZ(t, \abs\zt) - \zt = 0
     \end{defarray}
     \\
     &\hphantom:=
     \defset{(t, \zt(t))}
     {t \in \Domt,\ \cE(t, \abs{\zt(t)}) = 0,\ \cI(t, \abs{\zt(t)}) \ge 0}.
   \end{align*}
\iffalse
The feasible set of \eqref{eq:i-anf} can be rewritten using $\zt(t)$:
\begin{equation*}
 \Fabs = \{ (t, \zt(t)) : t\in\Domt,\ \cE(t, \abs\zt(t)) = 0,\ \cI(t, \abs\zt(t)) \ge 0  \}.
\end{equation*}
\fi
% With this in mind, $t$ is called a feasible point (short-hand notation $t\in\Fabs$) if $(t,\zt(t))\in\Fabs$.

In contrast to standard NLP theory,
we do not count equalities as active constraints.

\begin{definition}[Active Inequality Set]
  Let $(t,\zt(t))\in \Fabs$. We call the constraint
  $i\in\setI$ \emph{active} if $c_i(t,\abs{\zt(t)})=0$.
  The \emph{active set} $\setA(t)$ consists of all indices of active constraints,
  \begin{equation*}
    \setA(t)=\defset{i\in\setI}{c_i(t,\abs{\zt(t)})=0}.
  \end{equation*}
  We denote the number of active inequality constraints by $\card{\setA(t)}$.
  % and the number of inactive ones by $\card{\setA(t)^c}$. % Unused.
\end{definition}

To define the linear independence kink qualification as well as the interior direction kink qualification
for \eqref{eq:i-anf} we need its Jacobians.

\begin{definition}[Jacobians]
  \label{def:active-jacobian}
\ifcase1
  Let $\cE \in \Cone(\Domtzt,\R^{m_1})$, $\cI \in \Cone(\Domtzt,\R^{m_2})$ and
  $\cZ \in \Cone(\Domtzt,\R^{s_t})$.
\else
  Consider the abs-normal NLP \eqref{eq:i-anf}.
\fi
  For $(t,\zt(t)) \in \Fabs$ set $\setA = \setA(t)$, $\alpha = \alpha(t)$, $\sigma = \sigma(t)$,
  $\Sigma = \diag(\sigma)$, and $\cA = [c_i]_{i \in \setA}$.
  The \emph{equality-constraints Jacobian} is
  \begin{align*}
    \JE(t)
    \define
    \pd_t \cE(t,\Sigma \zt(t))
    &=
    \pd_1 \cE(t,\Sigma \zt(t)) +
    \pd_2 \cE(t,\Sigma \zt(t)) \Sigma \pd_t \zt(t) \\
    &=
    \pd_1 \cE(t,\abs{\zt(t)}) +
    \pd_2 \cE(t,\abs{\zt(t)}) \Sigma \pd_t \zt(t),
  \end{align*}
  the \emph{active inequality Jacobian} is
  \begin{align*}
    \JA(t)
    \define
    \pd_t \cA(t,\Sigma \zt(t))
    &=
    \pd_1 \cA(t,\Sigma \zt(t)) +
    \pd_2 \cA(t,\Sigma \zt(t)) \Sigma \pd_t \zt(t) \\
    &=
    \pd_1 \cA(t,\abs{\zt(t)}) +
    \pd_2 \cA(t,\abs{\zt(t)}) \Sigma \pd_t \zt(t),
  \end{align*}
  and the \emph{active switching Jacobian} is
  \begin{align*}
    \Jalp(t)
    \define
    \left[ e_i^T \pd_t \zt(t) \right]_{i \in \alpha}
    =
    \left[
      e_i^T [I - \pd_2 \cZ(t,\abs{\zt(t)}) \Sigma]^{-1}
      \pd_1\cZ(t,\abs{\zt(t)})
    \right]_{i \in \alpha}
    .
  \end{align*}
\end{definition}

\begin{definition}[Linear Independence Kink Qualification]
  \label{def:likq}
\ifcase1
  Let $\cE \in \Cone(\Domtzt,\R^{m_1})$, $\cI\in \Cone(\Domtzt,\R^{m_2})$ and
  $\cZ \in \Cone(\Domtzt,\R^{s_t})$.
\fi
  Let $(t,\zt(t))\in\Fabs$.
  We say that the \emph{linear independence kink qualification (LIKQ)}
  holds for \eqref{eq:i-anf} at $t$ if
   \begin{align*}
   \Jabs(t)=
   \begin{bmatrix}
     \JE(t)\\
     \JA(t)\\
     \Jalp(t)
   \end{bmatrix}
   =
   \begin{bmatrix}
     \pd_t \cE(t,\abs{\zt(t)}) \\
     \pd_t \cA(t,\abs{\zt(t)}) \\
     [e_i^T\pd_t \zt(t)]_{i\in \alpha}
   \end{bmatrix}
   \in \R^{(m_1 + \card{\setA} + \calp) \x n_t}
   \end{align*}
   has full row rank $m_1 + \card{\setA} + \calp$.
\end{definition}

\begin{definition}[Interior Direction Kink Qualification]
  \label{def:idkq}
\ifcase1
  Let $\cE \in \Cone(\Domtzt,\R^{m_1})$, $\cI \in \Cone(\Domtzt,\R^{m_2})$ and
  $\cZ \in \Cone(\Domtzt,\R^{s_t})$.
\fi
  Let $(t,\zt(t))\in\Fabs$.
  We say that the \emph{interior direction kink qualification (IDKQ)}
  holds for \eqref{eq:i-anf} at $t$ if
   \begin{align*}
   \begin{bmatrix}
     \JE(t)\\
     \Jalp(t)
   \end{bmatrix}
   =
   \begin{bmatrix}
     \pd_t \cE(t,\abs{\zt(t)}) \\
     [e_i^T\pd_t \zt(t)]_{i\in \alpha}
   \end{bmatrix}
   \in \R^{(m_1 + \calp) \x n_t}
   \end{align*}
   has full row rank $m_1 + \calp$ and if there exists a vector $d\in\R^{n_t}$ such that
   \begin{equation*}
     \JE(t)d=0, \quad
     \Jalp(t)d=0, \quad \text{and} \quad
     \JA(t)d>0.
   \end{equation*}
\end{definition}

For the general abs-normal NLP \eqref{eq:i-anf} considered here,
IDKQ actually generalizes MFCQ from the smooth case
and corresponds to MPCC-MFCQ, as we will show below.
We cannot use the canonical name MFKQ, however,
since Griewank and Walther have already defined MFKQ
as a different weakening of LIKQ in \cite{Griewank_Walther_2017}.
% where no constraints are present (except switching constraints).
We believe that other possible names like
``Abs-normal MFKQ'' or ``Constrained MFKQ''
would produce confusion rather than clarification
and hence suggest the descriptive name ``Interior Direction KQ''.

The following example from \cite{Scheel_Scholtes_2000} (converted from MPCC form to abs-normal NLP form) shows that IDKQ is weaker than LIKQ
in the presence of inequality constraints.

\begin{example}[IDKQ is weaker than LIKQ]\label{ex:idkqlikq}
     Consider the problem
     \begin{align*}
       \minst[t \in \R^3, \zt \in \R]{t_1 + t_2 - t_3}
       & t_1 + t_2 - \abs\zt = 0, \\
       & 4 t_1 - t_3 \ge 0, \\
       & 4 t_2 - t_3 \ge 0, \\
       & t_1 - t_2 - \zt = 0,
     \end{align*}
     with solution $t^*=(0,0,0)$ and $(\zt)^*=0$. We compute
     \begin{align*}
       \JA(t^*)=
       \begin{bmatrix}
        4 & 0 & -1 \\
        0 & 4 & -1
       \end{bmatrix},\quad
       \JE(t^*)=
       \begin{bmatrix}
        1 & 1 & 0
       \end{bmatrix},
       \qtextq{and}
       \Jalp(t^*)=
       \begin{bmatrix}
        1 & -1 & 0
       \end{bmatrix}.
     \end{align*}
     Here, LIKQ is not satisfied but IDKQ is satisfied with $d=(0,0,-1)$.
\end{example}

% =============================================================================
\subsection{Abs-Normal NLPs with Inequality Slacks}
\label{sec:anf-equalities}
% =============================================================================

In this paragraph we consider abs-normal NLPs with slack variables introduced for all
inequalities. We make use of the absolute value of a slack variable, an idea due to Griewank.
This results in a class of purely equality-constrained abs-normal NLPs,
which simplifies the derivation of optimality conditions under the LIKQ,
see \cite{Hegerhorst_Steinbach:2019} and \cref{sec:opt-cond}.

Using slack variables $w\in \R^{m_2}$,
we obtain the following reformulation of \eqref{eq:nlp}:
\begin{align*}
  \minst[t, w]{f(t)}
  & g(t) = 0, \\
  & h(t) - \abs{w} = 0.
\end{align*}
Then, we express $g$ and $h$ in abs-normal form as in \eqref{eq:anf}
and introduce additional switching variables $\zw$ to handle $\abs{w}$.
This approach leads to the next definition.

\begin{definition}[Abs-Normal NLP with Inequality Slacks]
\ifcase1
  Let $\Domt$ be an open subset of $\R^{n_t}$.
  A non-smooth NLP is called \emph{abs-normal NLP with inequality slacks}
  if functions $f \in \Cd(\Domt,\R)$, $\cE \in \Cd(\Domtzt,\R^{m_1})$,
  $\cI \in \Cd(\Domtzt,\R^{m_2})$,   $\cZ \in \Cd(\Domtzt,\R^{s_t})$
  with $d \ge 1$ exist such that the NLP reads
\or
  An abs-normal NLP posed in the following form is called
  an \emph{abs-normal NLP with inequality slacks}:
\fi
  \begin{align*}
    \minst[t, w, \zt, \zw]{f(t)}
    & \cE(t, \abs\zt) = 0, \\
    & \cI(t, \abs\zt) - \abs\zw = 0, \tag{E-NLP} \label{eq:e-anf} \\
    & \cZ(t, \abs\zt) - \zt = 0, \\
    & w-\zw = 0,
  \end{align*}
  where $\Domzt$ is open and symmetric and
  $\pd_2 \cZ(x,\abs{\zt})$ is strictly lower triangular.
\end{definition}
  The feasible set of \eqref{eq:e-anf} is a lifting of $\Fabs$,
  \begin{align*}
    \Feabs
    &\define
    \begin{defarray}{(t, w, \zt, \zw)}
      \cE(t, \abs\zt) = 0,\ \cI(t, \abs\zt) - \abs\zw = 0, \\
      \cZ(t, \abs\zt) - \zt = 0,\ w - \zw = 0
    \end{defarray}
    \\
    &\hphantom:=
    \Defset{(t, w, \zt, \zw)}
    {(t, \zt) \in \Fabs,\ w = \zw,\ \abs\zw = \cI(t, \abs\zt)}.
  \end{align*}

Using the dependence of $\zt$ and $\zw$ of $t$ and $w$, the feasible set can be written as
  \begin{align*}
    \Feabs
    &=
    \begin{defarray}{(t, w, \zt(t), \zw(w))}
      t\in\Domt,\ \cE(t, \abs{\zt(t)}) = 0,\\
      \cI(t, \abs{\zt(t)}) - \abs{\zw(w)} = 0,\ w - \zw(w) = 0
    \end{defarray}
    \\
    &=
    \Defset{(t, w, \zt(t), \zw(w))}
    {(t,\zt(t)) \in \Fabs,\ w = \zw(w),\ \abs{\zw(w)} = \cI(t, \abs{\zt(t)})}.
  \end{align*}

We split the active switching set into subsets for variables $t$ and $w$
as $\alpha = (\alpt, \alpw)$.

\begin{remark}
  Note that introducing $\abs{w}$ converts inequalities
  to pure equalities without a nonnegativity condition
  for the slack variables $w$.
\ifcase1
  In \cite{Hegerhorst_Steinbach:2019} we have used this formulation
  to simplify the presentation of first and second order conditions
  for the general non-smooth NLP \eqref{eq:nlp}
  under the linear independence kink qualification (LIKQ).
\fi
  However, the slack reformulation has some subtle issues.
  Subsequently we will show that, in contrast to LIKQ,
  IDKQ is \emph{not} preserved.
  Moreover, one cannot eliminate
  the equation $w - \zw = 0$ (and hence $\zw$ or~$w$) in \eqref{eq:e-anf}
  since this would destroy the abs-normal form.
  Finally, the slack $w$ is not uniquely determined
  since the signs of nonzero components $w_i$ can be chosen arbitrarily,
  yielding a set of $2^{m_2 - \calpw}$ choices,
  $W(t) \define \defset{w}{\abs{w} = \cI(t, \abs{\zt(t)})}$.
\end{remark}

\ifcase1% 2nd part wrong, but structure of cones has to do with 1st part.
Clearly, $\alpw = \setA$, and hence $\Feabs$ consists of $2^{\setA^c}$
components of the form $\Fabs \x \set{(w,\zw)}$,
given by all sign combinations of $w_i = \zw_i$ for $i \in \setA^c$.
\fi

\begin{lemma}%[LIKQ for \eqref{eq:e-anf}]
  \label{lem:e-likq}
\ifcase1
  Let $\cE \in \Cone(\Domtzt,\R^{m_1})$, $\cI\in \Cone(\Domtzt,\R^{m_2})$ and
  $\cZ \in \Cone(\Domtzt,\R^{s_t})$.
  Then, LIKQ for \eqref{eq:e-anf} at $(t, w) \in \Domt \x \R^{m_2}$
\else
  Consider $(t, w,\zt(t),\zw(w)) \in \Feabs$. Then, LIKQ for \eqref{eq:e-anf} at $(t,w)$
\fi
  is full row rank of
  \begin{align*}
  \Jeabs(t,w)=
   \begin{bmatrix}
     \pd_t \cE(t,\abs{\zt(t)}) & 0 \\
     \pd_t \cI(t,\abs{\zt(t)}) & -\Sigma^w\\
     [e_i^T\pd_t \zt(t)]_{i\in \alpt} & 0 \\
     0 & [e_i^T I]_{i\in \alpw}
   \end{bmatrix}
   \in \R^{(m_1 + m_2 + \calpt + \calpw) \x (n_t+m_2)}.
   \end{align*}
\end{lemma}
\begin{proof}
Set $x=(t,w)$, $z=(\zt,\zw)$, $\_f(x)=f(t)$,
$\bcE(x,\abs{z})=(\cE(t,\abs{\zt}), \cI(t,\abs{\zt})-\abs{\zw})$, and
$\bcZ(x,\abs{z})=(\cZ(t,\abs{\zt}),w)$.
Then, we can write \eqref{eq:e-anf} compactly as
\begin{align*}
  \minst[x, z]{\_f(x)}
  & \bcE(x, \abs{z}) = 0, \tag{$\overline{\text{E-NLP}}$} \label{eq:be-nlp} \\
  & \bcZ(x, \abs{z}) - z = 0,
\end{align*}
and compute $\bJE$ and $\bJalp$ from \cref{def:active-jacobian} using the special
structure of \eqref{eq:e-anf}.
The resulting matrix $\Jeabs(x)=\begin{bmatrix} \bJE(x)^T & \bJalp(x)^T \end{bmatrix}^T$
in \cref{def:likq} has the form above.
\end{proof}

\begin{remark}
  \label{rem:e-likq}
  Clearly, the rank of $\Jeabs$ does not depend on the signs
  of $\pm1$ entries in $\Sigma^w$ but only on their positions.
  Hence, LIKQ does not depend on the particular choice of~$w$.
  Otherwise it would not make sense to consider \eqref{eq:e-anf}.
\end{remark}

Note that, since the abs-normal NLP \eqref{eq:e-anf}
does not contain any inequalities,
the concept of IDKQ is equivalent to LIKQ here.
This is in contrast to the standard reformulation
of smooth NLP inequalities as equalities with nonnegative slacks
where the validity of LICQ and MFCQ are both unaffected.

% =============================================================================
\subsection{Relations of Kink Qualifications for Abs-Normal NLPs}
% =============================================================================

In this paragraph we discuss the relations of kink qualifications
for the two different formulations of abs-normal NLPs.
We use the set $W(t)$ from above.

\begin{theorem}\label{th:likq}
 LIKQ for \eqref{eq:i-anf} holds at $(t,\zt(t))\in\Fabs$ if and only if LIKQ for \eqref{eq:e-anf} holds at
 $(t,w,\zt(t),\zw(w))\in\Feabs$ for any (and hence all) $w \in W(t)$.
\end{theorem}
\begin{proof}
This follows immediately by comparison of $\Jabs$ and $\Jeabs$ using that
\begin{equation*}
 \alpw(w)=\defset{i\in\setI}{w_i=0}=\defset{i\in\setI}{c_i(t,\zt(t))=0}=\setA(t)
\end{equation*}
and
\begin{equation*}
 \Sigma^w=\diag(\sigw) \quad \text{with} \quad
 \sigw_i=\sign(w_i)
 =\begin{cases} 0, &i\in\setA(t), \\ \pm 1, &i\notin\setA(t). \end{cases}
 \qedhere
\end{equation*}
\end{proof}

\begin{theorem}\label{th:idkq}
 IDKQ for \eqref{eq:i-anf} holds at $(t,\zt(t))\in\Fabs$ if
 IDKQ for \eqref{eq:e-anf} holds at \ifnum\FmtChoice=2 the lifted point \fi
  $(t,w,\zt(t),\zw(w))\in\Feabs$
 for any (and hence all) $w \in W(t)$.
 The converse is not true.
\end{theorem}
\begin{proof}
Since \eqref{eq:e-anf} has no inequalities,
the concepts of IDKQ and LIKQ coincide.
LIKQ for \eqref{eq:e-anf} is equivalent to
LIKQ for \eqref{eq:i-anf} by \cref{th:likq}, and
LIKQ for \eqref{eq:i-anf} implies IDKQ for \eqref{eq:i-anf}.
The converse does not hold since LIKQ for \eqref{eq:i-anf}
is stronger then IDKQ as we have shown in \cref{ex:idkqlikq}.
\end{proof}

% Introducing slacks for inequality constraints in nonlinear optimization problems typically improves regularity of the formulation. For example, a regular basis of active inequality constraints becomes available.
% Contrary to this, Lemma~\cref{lem:akq} means that the absolute-value slack reformulation \eqref{eq:e-anf} potentially is \emph{less} regular than the inequality constrained one \eqref{eq:i-anf}.
% Hence, Griewank-type absolute-value slacks are not such a great idea.

%%%%%%%%%%%%%%%%%%%%%%%%%%%%%%%%%%%%%%%%%%%%%%%%%%%%%%%%%%%%%%%%%%%%%%%%%%%%%%%%%%
\section{Counterpart MPCCs}\label{sec:mpcc}
%%%%%%%%%%%%%%%%%%%%%%%%%%%%%%%%%%%%%%%%%%%%%%%%%%%%%%%%%%%%%%%%%%%%%%%%%%%%%%%%%%

In this section we introduce MPCC counterpart problems for the two formulations
\eqref{eq:i-anf} and \eqref{eq:e-anf}. Then, we have a quick look at relations between them.

% ============================================================================
\subsection{Counterpart MPCC for the General Abs-Normal NLP}\label{subsec:i-mpcc}
% ============================================================================

To reformulate \eqref{eq:i-anf} as an MPCC, we partition $\zt$ into its nonnegative part and the modulus of its
nonpositive part, $\ut \define [\zt]^+\define\max(\zt,0)$ and $\vt \define[\zt]^-\define \max(-\zt,0)$.
Then, we require complementarity of these two variables
to replace $\abs{\zt}$ by $\ut+\vt$ and $\zt$ itself by $\ut-\vt$.

\begin{definition}[Counterpart MPCC of \eqref{eq:i-anf}]
  The \emph{counterpart MPCC} of the abs-normal NLP \eqref{eq:i-anf} reads
   \begin{subequations}
   \begin{align*}
     \minst[t, \ut, \vt]{f(t)}
     & \cE(t, \ut + \vt) = 0, \\
     & \cI(t, \ut + \vt) \ge 0, \tag{I-MPCC} \label{eq:i-mpcc} \\
     & \cZ(t, \ut + \vt) - (\ut - \vt) = 0, \\
     & 0 \le \ut \compl \vt \ge 0,
   \end{align*}
   where $\ut, \vt\in\R^{s_t}$.
   \end{subequations}
\end{definition}
   The feasible set of \eqref{eq:i-mpcc} is denoted by
   \begin{equation*}
     \Fmpcc \define
     \begin{defarray}{(t, \ut, \vt)}
       \cE(t, \ut + \vt) = 0,\ \cI(t, \ut + \vt) \ge 0, \\
       \cZ(t, \ut + \vt) = \ut - \vt,\ 0 \le \ut \compl \vt \ge 0
     \end{defarray}
     .
   \end{equation*}

\begin{lemma}
  \label{lem:hom-F-i}
  Given an abs-normal NLP \eqref{eq:i-anf}
  and its counterpart MPCC \eqref{eq:i-mpcc},
  we have a homeomorphism $\phi\: \Fmpcc \to \Fabs$ defined as
  \begin{align*}
    \phi(t, \ut, \vt) &= (t, \ut - \vt), &
    \Inv\phi(t, \zt) &= (t, [\zt]^+, [\zt]^-).
  \end{align*}
\end{lemma}

\begin{proof}
  Obvious.
\end{proof}

\pagebreak[2]

Just like the active switching set of the abs-normal NLP, we define index sets of the counterpart MPCC.

\begin{definition}[Index Sets]
  We denote by $\setUt_0\define\defset{i\in \set{1, \dots, s_t}}{\ut_i=0}$
  the set of indices of active inequalities $\ut_i\geq 0$,
  and by $\setUt_+\define\defset{i\in\set{1, \dots, s_t}}{\ut_i>0}$
  the set of indices of inactive inequalities $\ut_i\geq 0$.
  Analogous definitions hold of $\setVt_0$ and $\setVt_+$.
  By $\setDt\define\setUt_0\cap\setVt_0$ we denote
  the set of indices of non-strict (degenerate) complementarity pairs.
  Thus we have the partitioning
  $\set{1, \dots, s_t} = \setUt_+ \cup \setVt_+ \cup \setDt$.
\end{definition}

In the following we define constraint qualifications for the counterpart MPCC.
The standard definitions say that MPCC-LICQ and MPCC-MFCQ are LICQ and MFCQ,
respectively, for the so-called \emph{tightened NLP} (see \cite{Scheel_Scholtes_2000})
with associated Jacobian
\begin{align*}
  J(t, \ut, \vt) =
  \def\arraystretch{1.3}
  \begin{bmatrix}
    \pd_1 \cE &
    \pd_2 \cE \PUt+^T & \pd_2 \cE \PUt0^T &
    \pd_2 \cE \PVt+^T & \pd_2 \cE \PVt0^T \\
    \pd_1 \cA &
    \pd_2 \cA \PUt+^T & \pd_2 \cA \PUt0^T &
    \pd_2 \cA \PVt+^T & \pd_2 \cA \PVt0^T \\
    \pd_1 \cZ &
    [\pd_2 \cZ - I] \PUt+^T & [\pd_2 \cZ - I] \PUt0^T &
    [\pd_2 \cZ + I] \PVt+^T & [\pd_2 \cZ + I] \PVt0^T \\
    0 & 0 & I & 0 & 0 \\
    0 & 0 & 0 & 0 & I
  \end{bmatrix},
\end{align*}
where $P_\setS \in \R^{\card\setS \x s_t}$ denotes the projector
onto the subspace defined by $\setS \subseteq \set{1, \dots, s_t}$
and all partial darivatives are evaluated at $(t, \ut + \vt)$.
This Jacobian will be needed in \cref{sec:soc}
to formulate second order conditions.
Here we exploit the two unit blocks
to state constraint qualifications in a more compact form as in \cite{Hegerhorst_et_al:2019:MPEC1}.

\begin{definition}[MPCC-LICQ for \eqref{eq:i-mpcc}, see \cite{Scheel_Scholtes_2000}]\label{def:mpcc-licq}
\ifcase1
  Let $\cE \in \Cone(\Domtzt,\R^{m_1})$, $\cI\in \Cone(\Domtzt,\R^{m_2})$ and
  $\cZ \in \Cone(\Domtzt,\R^{s_t})$. Consider a feasible point $(t,\ut,\vt)$ of \eqref{eq:i-mpcc}.
  We say that the \emph{MPCC-LICQ}
  holds for \eqref{eq:i-mpcc} at $(t,\ut,\vt)$ if
\else
  We say that the \emph{MPCC-LICQ} holds for \eqref{eq:i-mpcc}
  at a feasible point $(t,\ut,\vt)$ if
\fi
   \begin{align*}
   \Jmpcc(t,\ut,\vt) &=
   \def\arraystretch{1.3}
   \begin{bmatrix}
     \pd_1 \cE & \pd_2 \cE \PUt+^T & \pd_2 \cE \PVt+^T \\
     \pd_1 \cA & \pd_2 \cA \PUt+^T & \pd_2 \cA \PVt+^T \\
     \pd_1 \cZ & [\pd_2 \cZ - I] \PUt+^T & [\pd_2 \cZ + I] \PVt+^T
   \end{bmatrix}
   \\
   &\in\R^{(m_1+\card{\setA}+s_t)\x(n_t+\card{\setUt_+}+\card{\setVt_+})}
   \end{align*}
   has full row rank $m_1+\card{\setA}+s_t$.
   Here all partial derivatives are evaluated at $(t, \ut + \vt)$.
\end{definition}

\begin{definition}[MPCC-MFCQ for \eqref{eq:i-mpcc}, see \cite{Scheel_Scholtes_2000}]\label{def:mpcc-mfcq}
\ifcase1
  Let $\cE \in \Cone(\Domtzt,\R^{m_1})$, $\cI \in \Cone(\Domtzt,\R^{m_2})$ and
  $\cZ \in \Cone(\Domtzt,\R^{s_t})$. Consider a feasible point $(t,\ut,\vt)$ of \eqref{eq:i-mpcc}.
  We say that the \emph{MPCC-MFCQ}
  holds for \eqref{eq:i-mpcc} at $(t,\ut,\vt)$ if
\else
  We say that the \emph{MPCC-MFCQ} holds for \eqref{eq:i-mpcc}
  at a feasible point $(t,\ut,\vt)$ if
\fi
   \begin{align*}
     \def\arraystretch{1.3}
   \begin{bmatrix}
     \pd_1 \cE & \pd_2 \cE \PUt+^T & \pd_2 \cE \PVt+^T \\
     \pd_1 \cZ & [\pd_2 \cZ - I] \PUt+^T & [\pd_2 \cZ + I] \PVt+^T
   \end{bmatrix}
   \in \R^{(m_1+s_t)\x(n_t+\card{\setUt_+}+\card{\setVt_+})}
   \end{align*}
   has full row rank $m_1+s_t$ and if there exists a vector $d\in \R^{n_t+\card{\setUt_+}+\card{\setVt_+}}$ such that
   \begin{align*}
     \def\arraystretch{1.3}
     \begin{bmatrix}
     \pd_1 \cE & \pd_2 \cE \PUt+^T & \pd_2 \cE \PVt+^T \\
     \pd_1 \cZ & [\pd_2 \cZ - I] \PUt+^T & [\pd_2 \cZ + I] \PVt+^T
     \end{bmatrix}
     d&=0,\\
     \begin{bmatrix}
     \pd_1 \cA & \pd_2 \cA \PUt+^T & \pd_2 \cA \PVt+^T
     \end{bmatrix}
     d&>0.
   \end{align*}
   Again all partial derivatives are evaluated at $(t, \ut + \vt)$.
\end{definition}

As with LIKQ and IDKQ for \eqref{eq:i-anf}, MPCC-MFCQ is weaker then MPCC-LICQ for the counterpart MPCC of
\eqref{eq:i-anf}.
The latter fact is well known, and can also be seen easily by rewriting \cref{ex:idkqlikq} as the counterpart MPCC and checking the above
conditions.

% ============================================================================
\subsection{Counterpart MPCC for the Abs-Normal NLP with Inequality Slacks}
% ============================================================================

Using the same approach as in the preceding paragraph, we formulate the counterpart MPCC of \eqref{eq:e-anf}.

\begin{definition}[Counterpart MPCC of \eqref{eq:e-anf}]
  The \emph{counterpart MPCC} of the abs-normal NLP \eqref{eq:e-anf} reads:
   \begin{subequations}
   \begin{align*}
     \minst[t, w, \ut, \vt, \uw, \vw]{f(t)}
     & \cE(t, \ut + \vt) = 0, \\
     & \cI(t, \ut + \vt) - (\uw + \vw) = 0, \\
     & \cZ(t, \ut + \vt) - (\ut - \vt) = 0, \tag{E-MPCC} \label{eq:e-mpcc} \\
     & w - (\uw - \vw) = 0, \\
     & 0 \le \ut \compl \vt \ge 0,\\
     & 0 \le \uw \compl \vw \ge 0,
   \end{align*}
   \end{subequations}
   where $\ut, \vt \in \R^{s_t}$ and $\uw, \vw \in \R^{m_2}$.
\end{definition}
   The feasible set of \eqref{eq:e-mpcc} is a lifting of $\Fmpcc$:
   \begin{align*}
     \Fempcc &\define
     \begin{defarray}{(t, w, \ut, \vt, \uw, \vw)}
       \cE(t,\ut+\vt) = 0,\ \cI(t,\ut+\vt) = \uw+\vw, \\
       \cZ(t,\ut+\vt) = \ut-\vt,\ w = \uw-\vw, \\
       0 \le \ut \compl \vt \ge 0,\ 0 \le \uw \compl \vw \ge 0
     \end{defarray}
     \\ &\hphantom:=
     \begin{defarray}{(t, w, \ut, \vt, \uw, \vw)}
       (t, \ut, \vt) \in \Fmpcc,\ \cI(t,\ut+\vt) = \uw+\vw, \\
       w = \uw-\vw,\ 0 \le \uw \compl \vw \ge 0
     \end{defarray}
     .
   \end{align*}

Clearly, the homeomorphism between $\Fmpcc$ and $\Fabs$
extends to $\Fempcc$ and $\Feabs$.

\begin{lemma}
  \label{lem:hom-F-e}
  Given an abs-normal NLP \eqref{eq:e-anf}
  and its counterpart MPCC \eqref{eq:e-mpcc},
  we have a homeomorphism $\_\phi: \Fempcc \to \Feabs$ defined as
  \begin{align*}
    \_\phi(t, w, \ut, \vt, \uw, \vw) &= (t, w, \ut - \vt, \uw - \vw), \\
    \Inv{\_\phi}(t, w, \zt, \zw) &= (t, w, [\zt]^+, [\zt]^-, [\zw]^+, [\zw]^-).
  \end{align*}
\end{lemma}

\begin{proof}
  Obvious.
\end{proof}

\begin{lemma}%[MPCC-LICQ for \eqref{eq:e-mpcc}]
  \label{lem:e-mpcc_licq}
\ifcase1
  Let $\cE \in \Cone(\Domtzt,\R^{m_1})$, $\cI \in \Cone(\Domtzt,\R^{m_2})$ and
  $\cZ \in \Cone(\Domtzt,\R^{s_t})$. Then, MPCC-LICQ for \eqref{eq:e-mpcc}
  at $y = (t,w,\ut,\vt,\uw,\vw)$ is full row rank of
\else
  MPCC-LICQ for \eqref{eq:e-mpcc} at a feasible point
  $y = (t,w,\ut,\vt,\uw,\vw)$ is full row rank of
\fi
   \begin{align*}
   \Jempcc(y)&=
   \def\arraystretch{1.3}
   \begin{bmatrix}
     \pd_1 \cE & 0 & \pd_2 \cE \PUt+^T &
     \pd_2 \cE \PVt+^T & 0 & 0 \\
     \pd_1 \cI & 0 & \pd_2 \cI \PUt+^T
     & \pd_2 \cI \PVt+^T & -P_{\setU^w_+}^T & -P_{\setV^w_+}^T \\
     \pd_1 \cZ & 0 & [\pd_2 \cZ - I] \PUt+^T
     & [\pd_2 \cZ + I]\PVt+^T & 0 & 0 \\
     0 & I & 0 & 0 & -P_{\setU^w_+}^T & +P_{\setV^w_+}^T
   \end{bmatrix}
   \\
   &\in\R^{(m_1+m_2+s_t+m_2)\x(n_t+m_2+\card{\setUt_+}+\card{\setVt_+}+\card{\setU^w_+}+\card{\setV^w_+})},
   \end{align*}
   where all partial derivatives are evaluated at $(t, \ut + \vt)$.
\end{lemma}
\begin{proof}
We set $x=(t,w)$, $u=(\ut,\uw)$, $v=(\vt,\vw)$ as well as $\_f(x)=f(t)$,
\begin{equation*}
  \bcE(x,u+v)=
  \begin{pmatrix}\cE(t,\ut+\vt) \\ \cI(t,\ut+\vt)-(\uw+\vw))\end{pmatrix}
  ,\quad%\qtextq{and}
  \bcZ(x,u+v) = \begin{pmatrix}\cZ(t,\ut+\vt) \\ w\end{pmatrix}
  .
\end{equation*}
Then, \eqref{eq:e-mpcc} becomes
\begin{align*}
  \minst[x, u, v]{\_f(x)}
  & \bcE(x, u + v) = 0, \\
  & \bcZ(x, u + v) - (u - v) = 0,
  \tag{$\overline{\text{E-MPCC}}$} \label{eq:be-mpcc} \\
  & 0 \le u \compl v \ge 0,
\end{align*}
and we can compute the Jacobian from \cref{def:mpcc-licq} using the special
structure of \eqref{eq:e-mpcc}.
The resulting matrix has the stated form, except that the last four columns
belong to variables $(\ut,\vt,\uw,\vw)$ rather than $(u,v) = (\ut,\uw,\vt,\vw)$.
\end{proof}

Like LIKQ for \eqref{eq:e-anf}, MPCC-LICQ for \eqref{eq:e-mpcc}
does not depend on the particular choice of~$w$,
and like IDKQ for \eqref{eq:e-anf},
the concept of MPCC-MFCQ for \eqref{eq:e-mpcc}
is equivalent to MPCC-LICQ since no inequalities
are present besides the complementarities.

% ==========================================================================
\subsection{Relations of MPCC Constraint Qualifications}
% ==========================================================================

In this paragraph we state the relations of constraint qualifications for the two different formulations
introduced in the previous paragraphs.
They follow from the results in the previous section and
in the two \emph{following} sections.
For an illustration see \cref{fig:relations} below.
We set $W(t, \ut, \vt) \define \defset{(w, \uw, \vw)}
{\abs{w} = \cI(t, \ut + \vt), \uw = [w]^+, \vw = [w]^-}$.

\begin{theorem}\label{th:licq}
  MPCC-LICQ for \eqref{eq:i-mpcc} holds at $(t,\ut,\vt)\in\Fmpcc$
  if and only if MPCC-LICQ for \eqref{eq:e-mpcc} holds
  at $(t,w,\ut,\uw,\vt,\vw)\in\Fempcc$ for any (and hence all) $(w, \uw, \vw) \in W(t, \ut, \vt)$.
\end{theorem}
\begin{proof}
  This follows directly from \cref{th:likq},
  \cref{th:likq-licq-i} and \cref{th:likq-licq-e}.
\end{proof}

\begin{theorem}\label{th:mfcq}
  MPCC-MFCQ for \eqref{eq:i-mpcc} holds at $(t,\ut,\vt)\in\Fmpcc$ if
  MPCC-MFCQ for \eqref{eq:e-mpcc} holds at $(t,w,\ut,\uw,\vt,\vw)\in\Fempcc$
  for any (and hence all) $(w, \uw, \vw) \in W(t, \ut, \vt)$.
  The converse is not true.
\end{theorem}
\begin{proof}
  This follows directly from \cref{th:idkq},
  \cref{th:idkq-mfcq-i} and \cref{co:idkq-mfcq-e}.
\end{proof}

%%%%%%%%%%%%%%%%%%%%%%%%%%%%%%%%%%%%%%%%%%%%%%%%%%%%%%%%%%%%%%%%%%%%%%%%%%%%%%%%%%
\section{Kink Qualifications and MPCC Constraint Qualifications}\label{sec:regularity}
%%%%%%%%%%%%%%%%%%%%%%%%%%%%%%%%%%%%%%%%%%%%%%%%%%%%%%%%%%%%%%%%%%%%%%%%%%%%%%%%%%

In this section we have a closer look at relations between abs-normal NLPs
and counterpart MPCCs in both formulations.

% =============================================================================
\subsection{Relations of General Abs-Normal NLP and MPCC}
% =============================================================================

Here we use the variables $x$ and $z$ instead of $t$ and $\zt$
to shorten notation because we do not consider inequality slacks.
Then the general abs-normal NLP \eqref{eq:i-anf} becomes:
\begin{align*}
  \minst[x, z]{f(x)}
  & \cE(x, \abs{z}) = 0, \\
  & \cI(x, \abs{z}) \ge 0, \\
  & \cZ(x, \abs{z}) - z = 0.
\end{align*}
The counterpart MPCC \eqref{eq:i-mpcc} reads:
\begin{align*}
  \minst[x, u, v]{f(x)}
  & \cE(x, u + v) = 0, \\
  & \cI(x, u + v) \ge 0, \\
  & \cZ(x, u + v) - (u - v) = 0, \\
  & 0 \le u \compl v \ge 0.
\end{align*}
We obtain the following relations of kink qualifications
and MPCC constraint qualifications.

\ifcase1
\begin{theorem}[Equivalence of LIKQ and MPCC-LICQ]\label{th:likq-licq-i}
  LIKQ for \eqref{eq:i-anf} holds at $x\in\Fabs$ if and only if
  MPCC-LICQ for \eqref{eq:i-mpcc} holds at
\or
\begin{theorem}[LIKQ $\iff$ MPCC-LICQ]\label{th:likq-licq-i}
  LIKQ for \eqref{eq:i-anf} at $x\in\Fabs$ is equivalent to
  MPCC-LICQ for \eqref{eq:i-mpcc} at
\fi
\ifcase1
   $(x,u,v)\in\Domx\x\R^s\x\R^s$ with $u=[z(x)]^+$ and $v=[-z(x)]^-$.
\else
   $(x,u,v)=(x,[z(x)]^+,[z(x)]^-) \in \Fmpcc$.
\fi
  \end{theorem}
  \begin{proof}
    Setting $y \define (x, u + v)$ and  $r \define m_1 + \card\setA + s$,
    MPCC-LICQ for the counterpart MPCC is
   \begin{align*}
     \rank
     \def\arraystretch{1.3}
   \begin{bmatrix}
     \pd_1 \cE(y) & \pd_2 \cE(y) \PU^T & \pd_2 \cE(y) \PV^T \\
     \pd_1 \cA(y) & \pd_2 \cA(y) \PU^T & \pd_2 \cA(y) \PV^T \\
     \pd_1 \cZ(y) & [\pd_2 \cZ(y) - I] \PU^T & [\pd_2 \cZ(y) + I] \PV^T
   \end{bmatrix} = r.
   \end{align*}
   By negating the second column and combining it with the third column,
   this is equivalent to
   \begin{align*}
     \rank
     \def\arraystretch{1.3}
   \begin{bmatrix}
     \pd_1 \cE(y) & -\pd_2 \cE(y) \Sigma \PUV^T \\
     \pd_1 \cA(y) & -\pd_2 \cA(y) \Sigma \PUV^T \\
     \pd_1 \cZ(y) & [I - \pd_2 \cZ(y) \Sigma] \PUV^T
   \end{bmatrix} = r
   \end{align*}
   and, by non-singularity of $I - \pd_2 \cZ(y) \Sigma$, to
   \begin{align*}
     \rank
     \def\arraystretch{1.3}
   \begin{bmatrix}
     \pd_1 \cE(y) & -\pd_2 \cE(y) \Sigma \PUV^T \\
     \pd_1 \cA(y) & -\pd_2 \cA(y) \Sigma \PUV^T \\
     [I - \pd_2 \cZ(y) \Sigma]^{-1} \pd_1 \cZ(y) & \PUV^T
   \end{bmatrix} = r.
   \end{align*}
   Next, we use the third row to eliminate the entries above
   $\PUV^T$ to obtain
   \begin{align*}
     \rank
   \begin{bmatrix}
     \pd_1 \cE(y) +
     \pd_2 \cE(y) \Sigma [I - \pd_2 \cZ(y) \Sigma]^{-1} \pd_1 \cZ(y) & 0 \\
     \pd_1 \cA(y) +
     \pd_2 \cA(y) \Sigma [I - \pd_2 \cZ(y) \Sigma]^{-1} \pd_1 \cZ(y) & 0 \\
     [I - \pd_2 \cZ(y) \Sigma]^{-1} \pd_1 \cZ(y) & \PUV^T
   \end{bmatrix}
   = r,
 \end{align*}
 which we can write with $u + v = \abs{z} = \abs{z(x)}$ as
 \begin{align*}
   \rank
   \begin{bmatrix}
     \pd_x \cE(x, \abs{z(x)}) & 0 \\
     \pd_x \cA(x, \abs{z(x)}) & 0 \\
     \pd_x z(x) & \PUV^T
   \end{bmatrix} = r.
 \end{align*}
 Finally, since $\alpha = \setD$ is the complement of $\setUV$,
 this is equivalent to
 \begin{align*}
   \rank
   \begin{bmatrix}
     \pd_x \cE(x, \abs{z(x)}) \\
     \pd_x \cA(x, \abs{z(x)}) \\
     [e_i^T \pd_x z(x)]_{i \in \alpha}
   \end{bmatrix}
   = m_1 + \card\setA + \calp,
 \end{align*}
 which is LIKQ for the abs-normal NLP.
\end{proof}

\ifcase1
\begin{theorem}[Equivalence of IDKQ and MPCC-MFCQ]\label{th:idkq-mfcq-i}
  IDKQ for \eqref{eq:i-anf} holds at $x\in\Fabs$ if and only if
  MPCC-MFCQ for \eqref{eq:i-mpcc} holds at
\or
\begin{theorem}[IDKQ $\iff$ MPCC-MFCQ]\label{th:idkq-mfcq-i}
  IDKQ for \eqref{eq:i-anf} at $x\in\Fabs$ is equiv\-a\-lent to
  MPCC-MFCQ for \eqref{eq:i-mpcc} at
\fi
\ifcase1
   $(x,u,v)\in\Domx\x\R^s\x\R^s$ with $u=[z(x)]^+$ and $v=[z(x)]^-$.
\else
   $(x,u,v)=(x,[z(x)]^+,[z(x)]^-) \in \Fmpcc$.
\fi
  \end{theorem}
  \begin{proof}
    Again with $y \define (x, u + v)$,
   MPCC-MFCQ for the counterpart MPCC is
   \begin{enumerate}
    \item full row rank of
    \begin{align*}
    \def\arraystretch{1.3}
    \begin{bmatrix}
      \pd_1 \cE(y) & \pd_2 \cE(y)\PU^T & \pd_2 \cE(y)\PV^T \\
      \pd_1 \cZ(y) & [\pd_2 \cZ(y)-I]\PU^T & [\pd_2 \cZ(y)+I]\PV^T
    \end{bmatrix}
    \in \R^{(m_1 + s) \x (n + \card{\setUV})}.
    \end{align*}
    As in the proof of \cref{th:likq-licq-i},
    this is seen to be full row rank of
    \begin{align*}
    \begin{bmatrix}
     \pd_x \cE(x,\abs{z(x)})  \\
     [e_i^T\pd_x z(x)]_{i\in \alpha}
   \end{bmatrix}
     \in \R^{(m_1 + \calp) \x n}.
   \end{align*}
    \item the existence of a vector
      $d = (\dx,\du,\dv) \in \R^{n + \card{\setUV}}$ such that
   \begin{align*}
   \def\arraystretch{1.3}
   \begin{bmatrix}
     \pd_1 \cE(y) & \pd_2 \cE(y)\PU^T & \pd_2 \cE(y)\PV^T \\
     \pd_1 \cZ(y) & [\pd_2 \cZ(y)-I]\PU^T & [\pd_2 \cZ(y)+I]\PV^T
   \end{bmatrix}d = 0, \\
   \begin{bmatrix}
     \pd_1 \cA(y) & \pd_2 \cA(y)\PU^T & \pd_2 \cA(y)\PV^T \\
   \end{bmatrix}d>0.
   \end{align*}
   \end{enumerate}
   We combine $d_u$ and $-d_v$ to $d_{uv} \in \R^{\card{\setUV}}$.
   Then this is equivalent to
   \begin{align*}
     \pd_1 \cE(y)d_x + \pd_2 \cE(y)\Sigma \PUV^Td_{uv}&=0, \\
     \pd_1 \cZ(y)d_x - [I-\pd_2 \cZ(y)\Sigma] \PUV^Td_{uv}&=0, \\
     \pd_1 \cA(y)d_x + \pd_2 \cA(y)\Sigma \PUV^Td_{uv}&>0.
   \end{align*}
   The second condition can be written as
   \begin{equation}
     \label{eq:dx-duv}
     [I-\pd_2 \cZ(y)\Sigma]^{-1}\pd_1 \cZ(y)d_x =  \PUV^Td_{uv}.
   \end{equation}
   Multiplying this by $P_\setD^T = P_\alpha^T$ yields
   \begin{equation*}
     \left[
       e_i^T [I - \pd_2 \cZ(y) \Sigma]^{-1} \pd_1 \cZ(y)
     \right]_{i \in \alpha} d_x =
     [e_i^T \pd_x z(x)]_{i\in\alpha} d_x = 0.
   \end{equation*}
   With $u + v = \abs{z} = \abs{z(x)}$,
   substituting the right-hand side of \eqref{eq:dx-duv}
   into the first and third condition finally gives
   \begin{align*}
     \pd_x \cE(x,\abs{z(x)})d_x &=0, \\
     [e_i^T \pd_x z(x)]_{i\in\alpha} d_x &=0, \\
     \pd_x \cA(x,\abs{z(x)})d_x &>0,
   \end{align*}
   which is IDKQ for the abs-normal NLP.
  \end{proof}

% =============================================================================
\subsection{Relations of Abs-Normal NLP and MPCC with Inequality Slacks}
% =============================================================================

As the reformulation with inequality slacks
is just a specialization of the general case,
we do without proofs and give remarks where differences occur.

Using the short notation \eqref{eq:be-nlp} for \eqref{eq:e-anf}
(see proof of \cref{lem:e-likq})
and similarly \eqref{eq:be-mpcc} for the counterpart MPCC \eqref{eq:e-mpcc}
(see proof of \cref{lem:e-mpcc_licq}),
we obtain the same relation between LIKQ and MPCC-LICQ
as in the previous paragraph.

\ifcase1
\begin{theorem}[Equivalence of LIKQ and MPCC-LICQ]\label{th:likq-licq-e}
  LIKQ for \eqref{eq:e-anf} holds at $x\in\Feabs$ if and only if
  MPCC-LICQ for \eqref{eq:e-mpcc} holds at
\or
\begin{theorem}[LIKQ $\iff$ MPCC-LICQ]\label{th:likq-licq-e}
  LIKQ for \eqref{eq:e-anf} at $x\in\Feabs$ is equiv\-a\-lent to
  MPCC-LICQ for \eqref{eq:e-mpcc} at
\fi
\ifcase1
   $(x,u,v)\in\Domx\x\R^{s_t+m_2}\x\R^{s_t+m_2}$ with $u=[z(x)]^+$ and $v=[z(x)]^-$.
\else
   $(x,u,v)=(x,[z(x)]^+,[z(x)]^-) \in \Fempcc$.
\fi
\end{theorem}
\begin{proof}
  This follows as in the proof of \cref{th:likq-licq-i}.
\end{proof}

Note that this directly implies the next result since LIKQ and IDKQ as well as MPCC-LICQ and MPCC-MFCQ coincide in the purely equality constrained setting.

\ifcase1
\begin{corollary}[Equivalence of IDKQ and MPCC-MFCQ]\label{co:idkq-mfcq-e}
  IDKQ for \eqref{eq:e-anf} holds at $x\in\Feabs$ if and only if
  MPCC-MFCQ for \eqref{eq:e-mpcc} holds at
\or
\begin{corollary}[IDKQ $\iff$ MPCC-MFCQ]\label{co:idkq-mfcq-e}
  IDKQ for \eqref{eq:e-anf} at $x\in\Feabs$ is equiv\-a\-lent to
  MPCC-MFCQ for \eqref{eq:e-mpcc} at
\fi
\ifcase1
   $(x,u,v)\in\Domx\x\R^{s_t+m_2}\x\R^{s_t+m_2}$ with $u=[z(x)]^+$ and $v=[z(x)]^-$.
\else
   $(x,u,v)=(x,[z(x)]^+,[z(x)]^-) \in \Fempcc$.
\fi
\end{corollary}

\begin{figure}[tp]
    \begin{center}{
        \begin{tikzpicture}[ampersand replacement=\&]
            \matrix (m) [matrix of math nodes,nodes={draw,anchor=center,text centered,align=center,text width=3.2cm,rounded corners,minimum width=2cm, minimum height=1.5cm},column sep=12em,row sep=6em]
            {
            |(A)| {\begin{array}{c} \textup{Reformulation} \\ \textup{\eqref{eq:e-anf}} \end{array}}  \& |(D)|  {\begin{array}{c} \textup{Reformulation} \\ \textup{\eqref{eq:e-mpcc}} \end{array}}   \\
            |(B)| {\begin{array}{c} \textup{\kern-3pt Abs-normal NLP} \\
                \textup{\eqref{eq:i-anf}} \end{array}}  \& |(E)| {\begin{array}{c} \textup{General MPCC} \\ \textup{\eqref{eq:i-mpcc}} \end{array}}   \\
            };
            \begin{scope}[>={LaTeX[width=2mm,length=2mm]},->,line width=.5pt]
                \tikzstyle{every node}=[font=\footnotesize]
                \path[dashed] ([yshift=1em]A.east) edge[<->] node[below]{\cref{co:idkq-mfcq-e}}([yshift=1em]D.west);
                \path[] ([yshift=-1em]A.east) edge[<->] node[below] {\cref{th:likq-licq-e}} ([yshift=-1em]D.west);

                \path[dashed] ([xshift=-2em]A.south) edge[->] node[above,rotate=-90]{\cref{th:idkq}}([xshift=-2em]B.north);
                \path[] ([xshift=2em]A.south) edge[<->] node[above,rotate=-90]{\cref{th:likq}}([xshift=2em]B.north);

                \path[] ([xshift=-2em]D.south) edge[<->] node[above,rotate=90]{(implied)}([xshift=-2em]E.north);
                \path[dashed] ([xshift=2em]D.south) edge[->] node[above,rotate=90]{(implied)}([xshift=2em]E.north);

                \path[] ([yshift=1em]B.east) edge[<->] node[above] {\cref{th:likq-licq-i}} ([yshift=1em]E.west);
                \path[dashed] ([yshift=-1em]B.east) edge[<->]  node[above] {\cref{th:idkq-mfcq-i}} ([yshift=-1em]E.west);
            \end{scope}
        \end{tikzpicture}}%}
    \end{center}
    \caption{Solid arrows: relations between LIKQ and MPCC-LICQ;
      dashed arrows: relations between IDKQ and MPCC-MFCQ.}
    \label{fig:relations}
\end{figure}
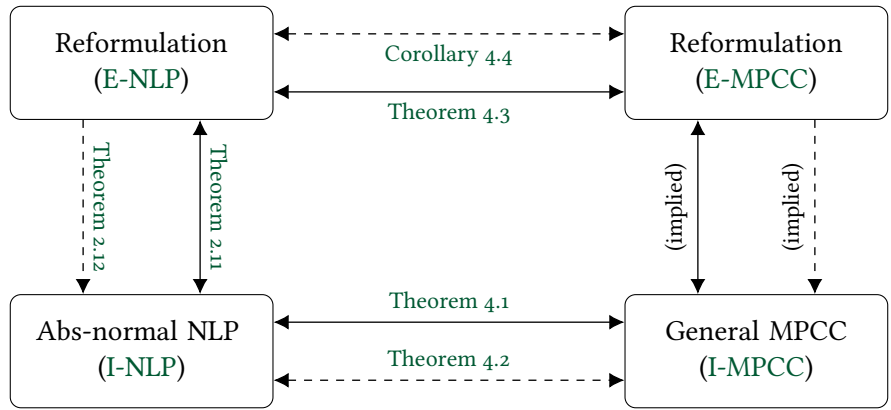

%================================================================================
\section{Optimality Conditions}
\label{sec:opt-cond}
%================================================================================

In this section we consider first and second order optimality conditions
for \eqref{eq:i-mpcc} under MPCC-LICQ and for \eqref{eq:i-anf} under LIKQ,
respectively, and discuss their relations.
Since both regularity conditions are invariant under the slack reformulation
by \cref{th:likq} and \cref{th:licq},
the results hold also for \eqref{eq:e-mpcc} and \eqref{eq:e-anf}.
Conditions for general MPCCs can be found in the literature;
in case of first order conditions for example in~\cite{Luo_Pang_Ralph_1996,Scheel_Scholtes_2000}.
Second order conditions stated in \cite{Luo_Pang_Ralph_1996,Scheel_Scholtes_2000}
however have to be adapted to our different setting.
For the abs-normal NLP \eqref{eq:e-anf} we have derived first and
second order conditions in \cite{Hegerhorst_Steinbach:2019}.
Since LIKQ is preserved under the slack reformulation by \cref{th:likq},
we can transfer these results directly to \eqref{eq:i-anf}.

%================================================================================
\subsection{First Order Optimality Conditions}
%================================================================================
In this paragraph, we compare stationarity concepts and first order optimality conditions
for \eqref{eq:i-mpcc} and \eqref{eq:i-anf}.
\ifcase1
\textcolor{orange}{
In this section, we compare strong stationarity for MPCCs to kink stationarity for abs-normal NLPs.
We give proofs based on the inequality constrained formulation and highlight differences to the equality constrained formulation where necessary.
}
\fi
First, we define strong stationarity for ~\eqref{eq:i-mpcc}
and state the corresponding first order conditions.

\begin{definition}[Strong Stationarity,
 see {\citep[\S3.3]{Luo_Pang_Ralph_1996}\ifnum\FmtChoice=1\ and \else, \fi
 \citep[Thm.~2]{Scheel_Scholtes_2000}}]
 \label{def:s-stat}
  A feasible point $y^*=(t^*,(\ut)^*,(\vt)^*)$ of \eqref{eq:i-mpcc}
  is \emph{strongly stationary} (\emph{S-stationary})
  if there exist multipliers $\lambda = (\lamE,\lamI,\lamZ)$ and
  $\mu=(\muu,\muv)$ such that the following conditions are satisfied:
  \begin{subequations}\label{eq:s-stat}
  \begin{align}
    \pd_y \Lc(y^*, \lambda, \mu) &= 0, \label{eq:s-stat-a} \\
    (\muu)_i \ge 0, \
    (\muv)_i & \ge 0, \quad i \in \setDt(t^*), \label{eq:s-stat-b} \\
    (\muu)_i &= 0, \quad i \in \setUt_+(t^*), \label{eq:s-stat-c} \\
    (\muv)_i &= 0, \quad i \in \setVt_+(t^*), \label{eq:s-stat-d} \\
    \lamI &\ge 0, \label{eq:s-stat-e} \\
    \lamI^T \cI(t^*, (\ut)^* + (\vt)^*) &= 0 \label{eq:s-stat-f}.
  \end{align}
  \end{subequations}
  Herein, $\Lc$ is the MPCC-Lagrangian function associated with \eqref{eq:i-mpcc}:
  \begin{align*}
    \Lc(y, \lambda, \mu) \define f(t)
    &+\lamE^T\cE(t,\ut + \vt) - \lamI^T\cI(t,\ut+\vt)\\
    &+\lamZ^T[\cZ(t,\ut+\vt) - (\ut - \vt)] - \muu^T \ut - \muv^T \vt.
  \end{align*}
\end{definition}

\begin{theorem}[First Order Optimality Conditions for \eqref{eq:i-mpcc}]
  \label{th:local-min-strong-stat}
\ifnum\FmtChoice=1
  Assume that the point $y^* = (t^*,(\ut)^*,(\vt)^*)$
  is a local minimizer of \eqref{eq:i-mpcc}
  and that MPCC-LICQ holds at $y^*$.
  Then,  $y^*$ is an S-stationary point.
\else
  Assume that $(t^*,(\ut)^*,(\vt)^*)$ is a local minimizer of \eqref{eq:i-mpcc}
  and that MPCC-LICQ holds at $(t^*,(\ut)^*,(\vt)^*)$.
  Then,  $(t^*,(\ut)^*,(\vt)^*)$ is an S-stationary point.
\fi
\end{theorem}
\begin{proof}
  The proof is due to \citep[\S3.3]{Luo_Pang_Ralph_1996} and is presented in \citep[Thm.~2]{Scheel_Scholtes_2000} in the form used here.
\end{proof}

Now, kink stationarity is defined and the corresponding first order optimality conditions are formulated.

\begin{definition}[Kink Stationarity, see \cite{Hegerhorst_Steinbach:2019}]\label{def:kink-stat}
  A feasible point $(t^*,(\zt)^*)$ of \eqref{eq:i-anf}
  is \emph{kink stationary}
  if there exist multipliers $\lambda = (\lamE,\lamI,\lamZ)$
  such that the following conditions are satisfied:
  \begin{subequations}
    \label{eq:kink-stat}
    \begin{align}
      \qquad f'(t^*) +
      \lamE^T \pd_1 \cE -
      \lamI^T \pd_1 \cI +
      \lamZ^T \pd_1 \cZ &= 0, \label{eq:kink-stat-a} \\
      [\lamE^T \pd_2 \cE -
      \lamI^T \pd_2 \cI +
      \lamZ^T \pd_2 \cZ]_i &\ge \abs{(\lamZ)_i},
      & i &\in \alpt(t^*), \qquad \label{eq:kink-stat-b} \\
      [\lamE^T \pd_2 \cE -
      \lamI^T \pd_2 \cI +
      \lamZ^T \pd_2 \cZ]_i &= (\lamZ)_i (\sigt)^*_i,
      & i &\notin \alpt(t^*), \label{eq:kink-stat-c} \\
      \lamI &\ge 0, \label{eq:kink-stat-d} \\
      \lamI^T \cI &= 0 \label{eq:kink-stat-e}.
    \end{align}
  \end{subequations}
  Here, the constraints and the partial derivatives are evaluated
  at $(t^*,\abs{(\zt)^*})$.
\end{definition}

\begin{theorem}[First Order Conditions for \eqref{eq:i-anf}]
  \label{th:local-min-kink-stat}
  Assume that $(t^*,(\zt)^*)$ is a local minimizer of \eqref{eq:i-anf}
  and that LIKQ holds at $t^*$.
  Then, $(t^*,(\zt)^*)$ is a kink stationary point.
\end{theorem}
\begin{proof}
  By \cref{th:likq} we may consider the slack reformulation
  \eqref{eq:e-anf} instead of \eqref{eq:i-anf}.
  In \cite[Theorem 5.10]{Hegerhorst_Steinbach:2019},
  conditions \eqref{eq:kink-stat} were proven for \eqref{eq:be-nlp}
  using a splitting of the switching variables $z$
  and the switching constraints $\cZ$.
  Without the splitting they read:
  \begin{align*}
    \qquad \_f'(x^*) +
    \blamE^T \pd_1 \bcE(x^*,\abs{z^*}) +
    \blamZ^T \pd_1 \bcZ(x^*,\abs{z^*}) &= 0, \\
    [\blamE^T \pd_2 \bcE(x^*,\abs{z^*}) +
    \blamZ^T \pd_2 \bcZ(x^*,\abs{z^*})]_i &\ge \abs{(\blamZ)_i},
    & i &\in \alpha(x^*), \qquad \\
    [\blamE^T \pd_2 \bcE(x^*,\abs{z^*}) +
    \blamZ^T \pd_2 \bcZ(x^*,\abs{z^*})]_i &= (\blamZ)_i \sigma^*_i,
    & i &\notin \alpha(x^*).
  \end{align*}
  We rewrite these conditions in the original notation of \eqref{eq:e-anf}
  with $\blamE = (\lamE, -\lamI)$ and $\blamZ = (\lamZ,\lamZw)$,
  where all derivatives are evaluated at $(t^*, \abs{(\zt)^*})$:
  \begin{align*}
    \qquad f'(t^*) +
    \lamE^T \pd_1 \cE -
    \lamI^T \pd_1 \cI +
    \lamZ^T \pd_1 \cZ &= 0, \\
    (\lamZw)^T &= 0, \\
    [\lamE^T \pd_2 \cE -
    \lamI^T \pd_2 \cI +
    \lamZ^T \pd_2 \cZ]_i &\ge \abs{(\lamZ)_i},
    & i &\in \alpt(t^*), \\
    [\lamE^T \pd_2 \cE -
    \lamI^T \pd_2 \cI +
    \lamZ^T \pd_2 \cZ]_i &= (\lamZ)_i (\sigt)_i^*,
    & i &\notin \alpt(t^*), \\
    \lamI &\ge \abs{(\lamZw)_i}, & i &\in \alpw(w^*), \\
    \lamI &= (\lamZw)_i \sigma_i^*, & i &\notin \alpw(w^*). \qquad
  \end{align*}
  The claim follows by eliminating $\lamZw = 0$
  and noting that $\alpw(w^*) = \setA(t^*)$.
\end{proof}

The next theorem shows that the two stationarity concepts coincide.

\begin{theorem}[S-Stationarity is Kink Stationarity]
  \label{thm:s-stat-is-kink-stat}
  A feasible point $(t^*,(\zt)^*)$ of \eqref{eq:i-anf} is kink stationary
  if and only if $(t^*,(\ut)^*,(\vt)^*)=(t^*,[\zt(t^*)]^+,[\zt(t^*)]^-)$
  of \eqref{eq:i-mpcc} is S-stationary.
\end{theorem}
\begin{proof}
  Comparison of the stationarity conditions
  of \eqref{eq:i-anf} and \eqref{eq:i-mpcc}
  shows directly that \eqref{eq:s-stat-e} and \eqref{eq:kink-stat-d}
  as well as \eqref{eq:s-stat-f} and \eqref{eq:kink-stat-e} coincide.
  Thus, we have to compare the remaining conditions
  \eqref{eq:s-stat-a} to \eqref{eq:s-stat-d} for \eqref{eq:i-mpcc}
  with \eqref{eq:kink-stat-a} to \eqref{eq:kink-stat-c} for \eqref{eq:i-anf}.
  Condition \eqref{eq:s-stat-a} of \eqref{eq:i-mpcc},
  where all derivatives are evaluated at $(t^*,(\ut)^* + (\vt)^*)$, is
  \begin{align*}
   f'(t^*) +
   \lamE^T \pd_1 \cE -
   \lamI^T \pd_1 \cI +
   \lamZ^T \pd_1 \cZ &= 0, \\
   \lamE^T \pd_2 \cE -
   \lamI^T \pd_2 \cI +
   \lamZ^T [\pd_2 \cZ - I] - \muu^T &= 0, \\
   \lamE^T \pd_2 \cE -
   \lamI^T \pd_2 \cI +
   \lamZ^T [\pd_2 \cZ + I] - \muv^T &= 0.
  \end{align*}
  The first condition coincides with \eqref{eq:kink-stat-a}.
  We combine the second and the third condition
  with conditions \eqref{eq:s-stat-b} to \eqref{eq:s-stat-d},
  yielding
  \begin{align*}
    \qquad
    \left[ \lamE^T \pd_2 \cE -
      \lamI^T \pd_2 \cI +
      \lamZ^T \pd_2 \cZ \hphantom{[{}+I]} \right]_i &= +(\lamZ)_i,
    & i &\in \setUt_+(t^*), \qquad \\
    \left[ \lamE^T \pd_2 \cE -
      \lamI^T \pd_2 \cI +
      \lamZ^T \pd_2 \cZ \hphantom{[{}+I]} \right]_i &= -(\lamZ)_i,
    & i &\in \setVt_+(t^*), \\
    \left[ \lamE^T \pd_2 \cE -
      \lamI^T \pd_2 \cI +
      \lamZ^T [\pd_2 \cZ \pm I] \right]_i & \ge 0,
    & i &\in \setDt(t^*).
  \end{align*}
  These are precisely conditions \eqref{eq:kink-stat-b}
  and \eqref{eq:kink-stat-c} for \eqref{eq:i-anf}
  by definition of the index sets and of~$\sigma^*$.
\end{proof}

As LIKQ for \eqref{eq:i-anf} is equivalent to MPCC-LICQ for \eqref{eq:i-mpcc},
the previous theorem provides a different perspective
on \cref{th:local-min-kink-stat}
and \cref{th:local-min-strong-stat}:
one can be obtained from the other directly
via \cref{thm:s-stat-is-kink-stat} and vice versa.

%================================================================================
\subsection{Second Order Conditions}\label{sec:soc}
%================================================================================

In this paragraph, we compare second-order conditions for MPCCs and abs-normal NLPs.

First, we formulate them for \eqref{eq:i-mpcc}.
This is based on material from \cite{Scheel_Scholtes_2000} for MPCCs, but some additional assumptions on the Lagrange multipliers need to be made.
These are given in the next definition.

\begin{definition}[MPCC-Strict Complementarity]
\ifnum\FmtChoice=1
  Given an S-stationary point $(t^*, (\ut)^*, (\vt)^*)$
  with Lagrange multipliers $(\lambda^*, \mu^*)$,
  we say that \emph{MPCC-strict complementarity} holds
\else
  Consider a strongly stationary point $(t^*, (\ut)^*, (\vt)^*)$
  with Lagrange multipliers $(\lambda^*, \mu^*)$.
  We say that \emph{MPCC-strict complementarity} holds
\fi
  if $\lambda_i^* > 0$ for all $i \in \setA$ as well as
  $(\muu^*)_i > 0$ and $(\muv^*)_i > 0$ for all $i \in \setDt$.
\end{definition}

We will show in the next lemma that under MPCC-LICQ and
MPCC-strict complementarity the critical cone reduces to
the nullspace of the Jacobian of the tightened NLP
(with columns reordered according to the index sets
$\setUt_+, \setUt_0, \setVt_+, \setVt_0$),
\begin{align*}
  J(y^*) =
  \begin{bmatrix}
    \pd_1 \cE &
    \pd_2 \cE \PUt+^T & \pd_2 \cE \PUt0^T &
    \pd_2 \cE \PVt+^T & \pd_2 \cE \PVt0^T \\
    \pd_1 \cA &
    \pd_2 \cA \PUt+^T & \pd_2 \cA \PUt0^T &
    \pd_2 \cA \PVt+^T & \pd_2 \cA \PVt0^T \\
    \pd_1 \cZ &
    [\pd_2 \cZ - I] \PUt+^T & [\pd_2 \cZ - I] \PUt0^T &
    [\pd_2 \cZ + I] \PVt+^T & [\pd_2 \cZ + I] \PVt0^T \\
    0 & 0 & I & 0 & 0 \\
    0 & 0 & 0 & 0 & I
  \end{bmatrix}
  ,
\end{align*}
as introduced in \cref{subsec:i-mpcc}.
Here, all partial derivatives are evaluated at the point $(t^*,(\ut)^*+(\vt)^*)$.
It is readily verified that the nullspace of $J(y^*)$ is spanned by the matrix
\begin{equation}
  \label{eq:Umpcc}
  \Umpcc(y^*) =
  \begin{bmatrix}
    I \\
    +\PUt+ (I - \pd_2 \cZ \Sigma^t)^{-1} \pd_1 \cZ \\ 0 \\
    -\PVt+ (I - \pd_2 \cZ \Sigma^t)^{-1} \pd_1 \cZ \\ 0
  \end{bmatrix}
  \tilde{U}(y^*)
\end{equation}
where $\tilde{U}(y^*)$ spans the nullspace of
\begin{equation}
  \label{eq:U1}
  \begin{bmatrix}
    \pd_1 \cE +
    \pd_2 \cE \Sigma^t (I - \pd_2 \cZ \ \Sigma^t)^{-1} \pd_1 \cZ \\
    \pd_1 \cA +
    \pd_2 \cA \Sigma^t (I - \pd_2 \cZ \ \Sigma^t)^{-1} \pd_1 \cZ \\
    [e_i^T (I - \pd_2 \cZ \Sigma^t)^{-1} \pd_1 \cZ]_{i \in \setDt}
  \end{bmatrix}
  .
\end{equation}

\ifcase0
Second order necessary and sufficient conditions
for a slightly more general class of MPCCs
are given in \cite[Theorem~7]{Scheel_Scholtes_2000}
using the concept of critical directions.
We first specialize the definition from \cite{Scheel_Scholtes_2000}
to our setting.
\or
Second order conditions can now be stated and will be proved
using the following definition and the first part of Theorem~7
from \cite{Scheel_Scholtes_2000} (given in our notation).
\fi

\begin{definition}[Critical Direction]
  A vector $d = (dt, du^t, dv^t) \in \R^{n_t} \x \R^{s_t} \x \R^{s_t}$
  is called a \emph{critical direction}
  at a weakly stationary point $y^*$ of \eqref{eq:i-mpcc} if
  \begin{subequations}
  \begin{align}
    \min(du^t_i, dv^t_i) &= 0, \quad i \in \setDt, \label{eq:cd-min} \\
    du^t_i &= 0, \quad i \in \setVt_+, \label{eq:cd-V+} \\
    dv^t_i &= 0, \quad i \in \setUt_+, \label{eq:cd-U+} \\
    \pd_1 \cA du^t + \pd_2 \cA (du^t + dv^t) &\ge 0, \label{eq:cd-A} \\
    \pd_1 \cE du^t + \pd_2 \cE (du^t + dv^t) &= 0, \label{eq:cd-E} \\
    \pd_1 \cZ du^t +
    [\pd_2 \cZ - I] du^t + [\pd_2 \cZ + I] dv^t &= 0, \label{eq:cd-Z} \\
    f'(t^*) dt & =0, \label{eq:cd-f'}
  \end{align}
  \end{subequations}
  where all constraint derivatives are evaluated at $(t^*, (\ut)^* + (\vt)^*)$.
\end{definition}

The set of critical directions is just
the nullspace of $J(y^*)$ under stronger assumptions.

\begin{lemma}
  Assume that MPCC-LICQ and MPCC-strict complementarity hold
  at an S-stationary point $y^*=(t^*,(\ut)^*,(\vt)^*)$ of \eqref{eq:i-mpcc}
  with Lagrange multipliers $(\lambda^*,\mu^*)$.
  Then, the set of critical directions is $\ker J(y^*)$.
\end{lemma}
\begin{proof}
  First consider a critical direction $d = (dt, du^t, dv^t)$
  at the strongly (hence weakly) stationary point $y^*$.
  Then, \eqref{eq:cd-E} and \eqref{eq:cd-Z} imply
  that rows one and three of $J(y^*) d$ vanish,
  and by \eqref{eq:s-stat-a} and \eqref{eq:cd-f'} we further have
  \begin{align*}
    0 &= \pd_{t,\ut,\vt} \Lc(y^*,\lambda^*,\mu^*) d \\
    &= f'(t^*) dt
    + (\lamE^*)^T [\pd_1 \cE dt + \pd_2 \cE (du^t + dv^t)] \\
    &\quad{} - (\lamI^*)^T [\pd_1 \cI dt + \pd_2 \cI (du^t + dv^t)] \\
    &\quad{} + (\lamZ^*)^T [\pd_1 \cZ dt +
    (\pd_2 \cZ - I) du^t + (\pd_2 \cZ + I) dv^t] \\
    &\quad{} - (\muu^*)^T du^t - (\muv^*)^T dv^t \\
    &= -(\lamI^*)^T [\pd_1 \cI dt + \pd_2 \cI (du^t + dv^t)]
    - (\muu^*)^T du^t - (\muv^*)^T dv^t.
  \end{align*}
  With $(\lamI^*)\tp \cI = 0$ \eqref{eq:s-stat-f},
  $(\muu^*)_i = 0$ for $i \in \setUt_+$ \eqref{eq:s-stat-c},
  $(\muv^*)_i = 0$ for $i \in \setVt_+$ \eqref{eq:s-stat-d},
  and \eqref{eq:cd-V+}, \eqref{eq:cd-U+} we obtain
  $(\lamI^*)_i = 0$ for $i \notin \setA$ and further
  \begin{equation*}
    0 =
    (\lamA^*)^T [\pd_1 \cA du^t + \pd_2 \cA (du^t + dv^t)]
    + \sum_{i \in \setDt} [(\muu^*)_i du^t_i + (\muv^*)_i dv^t_i].
  \end{equation*}
  All factors in this sum of products are nonnegative by
  \eqref{eq:s-stat-b}, \eqref{eq:s-stat-e},
  \eqref{eq:cd-A}, and \eqref{eq:cd-min}, which implies
  \begin{align*}
    0 &= (\lamA^*)^T [\pd_1 \cA du^t + \pd_2 \cA (du^t + dv^t)], \\
    0 &= (\muu^*)_i du_i^t = (\muv^*)_i dv_i^t, \quad i \in \setDt.
  \end{align*}
  Finally, by MPCC-strict complementarity we have
  \begin{align*}
    0 &= \pd_1 \cA du^t + \pd_2 \cA (du^t + dv^t), \\
    0 &= du^t_i = dv^t_i, \quad i \in \setDt,
  \end{align*}
  and $du^t_i = 0$ for $i \in \setUt_0$
  as well as $dv^t_i = 0$ for $i \in \setVt_0$
  follow since $\setUt_0 = \setVt_+ \cup \setDt$
  and $\setVt_0 = \setUt_+ \cup \setDt$.
  Thus $d$ is a nullspace vector of $J(y^*)$.

  Conversely, given any nullspace vector $d = (dt, du^t, dv^t)$,
  the first three rows of $J(y^*) d = 0$ yield conditions
  \eqref{eq:cd-E}, \eqref{eq:cd-A}, and \eqref{eq:cd-Z},
  with equality ``$=0$'' in case of \eqref{eq:cd-A}.
  The last two rows yield
  $du^t_i = 0$ for $i \in \setUt_0$ and
  $dv^t_i = 0$ for $i \in \setVt_0$,
  hence \eqref{eq:cd-V+}, \eqref{eq:cd-U+},
  and $du^t_i = dv^t_i = 0$ for $i \in \setDt$ \eqref{eq:cd-min}.
  Moreover, we have
  $(\muu^*)_i = 0$ for $i \in \setUt_+$ \eqref{eq:s-stat-c},
  $(\muv^*)_i = 0$ for $i \in \setVt_+$ \eqref{eq:s-stat-d},
  and $(\lamI^*)_i = 0$ for $i \notin \setA$ \eqref{eq:s-stat-f},
  so that \eqref{eq:s-stat-a} becomes \eqref{eq:cd-f'}:
  \begin{equation*}
    0 = \pd_{t,\ut,\vt} \Lc(y^*,\lambda^*,\mu^*) d = f'(t^*) dt.
  \end{equation*}
  Thus $d$ is a critical direction.
\end{proof}

Now we use \cite[Theorem~7]{Scheel_Scholtes_2000}
to prove second order necessary and sufficient conditions for our setting.

\ifcase1
\begin{theorem*}
  Assume that $y^* = (t^*, (\ut)^*, (\vt)^*)$ is a local minimizer
  of \eqref{eq:i-mpcc} and that MPCC-SMFCQ (see \cite[\S 3.2]{Luo_Pang_Ralph_1996}, \cite{Scheel_Scholtes_2000} and \cite[Def.~3.1]{FlegelDiss}) holds at $y^*$.
  Then, $y^*$ is a strongly stationary point
  with a unique Lagrange multiplier vector $(\lambda^*, \mu^*)$,
  and every critical direction $d$ satisfies the inequality
  \begin{align*}
    d^T \Hmpcc(y^*, \lambda^*) d \ge 0.
  \end{align*}
\end{theorem*}
\fi
\begin{theorem}[Second Order Necessary Conditions for \eqref{eq:i-mpcc}]
  Assume that \ifnum\FmtChoice=1 the point \fi
  $y^*=(t^*,(\ut)^*,(\vt)^*)$ is a local minimizer of \eqref{eq:i-mpcc}
  and that MPCC-LICQ holds at~$y^*$.
  Denote by $(\lambda^*,\mu^*)$ the unique Lagrange multiplier vector
  and assume further that MPCC-strict complementarity holds.
  Then,
  \begin{align*}
    \Umpcc(y^*)^T \Hmpcc(y^*, \lambda^*) \Umpcc(y^*) \ge 0
  \end{align*}
  where $\Hmpcc(y^*, \lambda^*) = \pd^2_{yy} \Lc(y^*, \lambda^*, \mu^*)$.
  (Note that $\pd^2_{yy} \Lc$ does not depend on $\mu^*$.)
\end{theorem}
\begin{proof}
  The first part of Theorem~7 in \cite{Scheel_Scholtes_2000} asserts that
  every critical direction $d$ satisfies the inequality
   $d^T \Hmpcc(y^*, \lambda^*) d \ge 0$
  at a local minimizer $y^*$ if MPCC-SMFCQ (cf.~\cite{Luo_Pang_Ralph_1996}) holds at $y^*$.
  Since MPCC-LICQ implies MPCC-SMFCQ
  and the set of critical directions
  is $\ker J(y^*)$ under our stronger assumptions,
  the claim follows directly from \cite[Theorem~7]{Scheel_Scholtes_2000}.
\end{proof}
\begin{remark}
  Here we have simplified the exposition by
  making the assumption of MPCC-strict complementarity,
  so that we can directly rely on \cite[Theorem~7]{Scheel_Scholtes_2000}.
  However, the second order necessary conditions
  can also be proved without MPCC-strict complementarity
  by considering branch problems of \eqref{eq:i-mpcc}.
  The corresponding approach for \eqref{eq:i-anf}
  has been taken in \cite{Hegerhorst_Steinbach:2019},
  so that \cref{th:2nd-nec-ianf} below
  does not require strict complementarity.
\end{remark}

\begin{theorem}[Second Order Sufficient Conditions for \eqref{eq:i-mpcc}]
\ifcase1
  Assume that $y^*=(t^*,(\ut)^*,(\vt)^*)$ is feasible for \eqref{eq:i-mpcc}
  and that MPCC-LICQ holds at $y^*$.
  Assume further that $y^*$ is strongly stationary
  with Lagrange multiplier vector $(\lambda^*,\mu^*)$
  satisfying MPCC-strict complementarity, and that
\else
  Assume that \ifnum\FmtChoice=1 the point \fi
  $y^* = (t^*, (\ut)^*, (\vt)^*)$ is strongly stationary
  for \eqref{eq:i-mpcc} with Lagrange multiplier vector $(\lambda^*,\mu^*)$
  satisfying MPCC-strict complementarity.
  Assume further that MPCC-LICQ holds at $y^*$, and that
\fi
  \begin{align*}
    \Umpcc(y^*)^T \Hmpcc(y^*, \lambda^*) \Umpcc(y^*) > 0.
  \end{align*}
  Then, $y^*$ is a strict local minimizer of ~\eqref{eq:i-mpcc}.
\end{theorem}
\begin{proof}
  In the second part of \cite[Theorem~7]{Scheel_Scholtes_2000},
  our assertion is proved under the weaker assumption that
  $y^*$ is strongly stationary and for every critical direction $d \ne 0$
  there exists a Lagrange multiplier vector $(\lambda^*, \mu^*)$ such that
  $d^T \Hmpcc(y^*, \lambda^*) d > 0$.
  Under our additional assumptions of MPCC-LICQ and MPCC-strict complementarity,
  the set of critical directions is spanned by the matrix $\Umpcc(y^*)$,
  see proof of the previous lemma.
  Thus, the claim follows directly from \cite[Theorem~7]{Scheel_Scholtes_2000}.
\end{proof}

\ifcase1
\begin{remark}\label{rem:hessian}
  All terms of Lagrange multipliers $\mu$ in $\Lc$ are just linear. Thus, $\Hmpcc(y^*,\lambda^*,\mu^*)$ does not depend on $\mu^*$ and we will write just $\Hmpcc(y^*,\lambda^*)$ in the following.
Moreover, the Hessian can be interpreted as the second derivative of
\begin{align*}
 \Lc(t,\ut,\vt,\lambda) \define & \Lc(t,\ut,\vt,\lambda,\mu)+\muu^T \ut + \muv^T \vt \\
 = & f(t) +\lamE^T\cE(t,\ut+\vt) - \lamI^T\cI(t,\ut+\vt) +\lamZ^T[\cZ(t,\ut+\vt) - (\ut - \vt)].
\end{align*}
\end{remark}
\fi

We proceed by formulating second-order conditions for \eqref{eq:i-anf}.
To this end, we denote by $(\alpt)^c$ the complement of $\alpt$,
and we need the Lagrangian
\begin{align*}
  \Lagr(t, \abs{\zt}, \lambda) = f(t) +
  \lamE^T \cE(t, \abs{\zt}) -
  \lamI^T \cI(t, \abs{\zt}) +
  \lamZ^T \bigl(
    \cZ(t, \abs\zt) - \setP_{(\alpt)^c}^T\setP_{(\alpt)^c} \Sigma^t \abs\zt \bigr)
\end{align*}
and the matrix
\begin{equation}
  \label{eq:Uabs}
  \Uabs(t) \define
  \begin{bmatrix}
    U(t) \\ [e_i^T \Sigma  \pd_t \zt(t) U(t)]_{i \notin \alpt}
  \end{bmatrix}
  ,
\end{equation}
where $U(t)$ spans the nullspace of $\Jabs(t)$.
We also use the Lagrangian
\begin{align*}
  \_\Lagr(x, \abs{z}, \_\lambda) = \_f(x)
  &+ \blamE^T \bcE(x, \abs{z}) + \blamZ^T \bigl(
  \bcZ(x, \abs{z}) - \setP_{\alpha^c}^T\setP_{\alpha^c} \Sigma \abs{z} \bigr), \\
  %\Lagr(t, w, \zt, \zw, \lambda)
  {}= f(t)
  &+ \lamE^T \cE(t, \abs\zt) - \lamI^T
  \bigl(\cI(t, \abs{\zt}) - \abs\zw \bigr) \\
  &+ \lamZ^T \bigl( \cZ(t, \abs\zt) -
  \setP_{(\alpt)^c}^T\setP_{(\alpt)^c} \Sigma^t \abs{\zt} \bigr) \\
  &+ (\lamZw)^T \bigl(
  w - \setP_{(\alpw)^c}^T\setP_{(\alpw)^c} \Sigma^w \abs{\zw} \bigr).
\end{align*}

\begin{theorem}[Second Order Necessary Conditions for \eqref{eq:i-anf}]
  \label{th:2nd-nec-ianf}
  Assume that $y^* = (t^*, (\zt)^*)$ is a local minimizer of \eqref{eq:i-anf}
  and that LIKQ holds at $t^*$.
  Denote by $\lambda^*$ the unique Lagrange multiplier and set $\alpt=\alpt(t^*)$. Then,
  \begin{equation*}
    \Uabs(t^*)^T \Habs(y^*, \lambda^*) \Uabs(t^*) \ge 0
  \end{equation*}
  where $\Habs(y^*, \lambda^*) =
    \begin{bmatrix}
      I & 0 \\
      0 & P_{(\alpt)^c}
    \end{bmatrix}
    \begin{bmatrix}
      \pd_{11} \Lagr(y^*, \lambda^*) & \pd_{12} \Lagr(y^*, \lambda^*) \\
      \pd_{21} \Lagr(y^*, \lambda^*) & \pd_{22} \Lagr(y^*, \lambda^*)
    \end{bmatrix}
    \begin{bmatrix}
      I & 0 \\
      0 & P_{(\alpt)^c}^T
    \end{bmatrix}$.
\end{theorem}
\begin{proof}
  As in \cref{th:local-min-kink-stat}, we can consider
  \eqref{eq:e-anf} instead of \eqref{eq:i-anf} by \cref{th:likq}.
  In \cite[Theorem 5.15]{Hegerhorst_Steinbach:2019}
  the second order necessary conditions for \eqref{eq:be-nlp}
  have been derived using a variable splitting.
  Without the splitting they read
  \begin{align*}
    \Ueabs(x^*)^T \Heabs(\_y^*, \_\lambda^*) \Ueabs(x^*) \ge 0
  \end{align*}
  with $\_y^* = (x^*, z^*)$, $\blamE = (\lamE, -\lamI)$,
  $\blamZ = (\lamZ,\lamZw)$, and the Hessian
  \begin{align*}
    \Heabs(\_y^*, \_\lambda^*) =
    \begin{bmatrix}
      I & 0 \\ 0 & P_{\alpha^c}
    \end{bmatrix}
    \begin{bmatrix}
      \pd_{11} \_\Lagr(\_y^*, \_\lambda^*) &
      \pd_{12} \_\Lagr(\_y^*, \_\lambda^*) \\
      \pd_{21} \_\Lagr(\_y^*, \_\lambda^*) &
      \pd_{22} \_\Lagr(\_y^*, \_\lambda^*)
    \end{bmatrix}
    \begin{bmatrix}
      I & 0 \\ 0 & P_{\alpha^c}^T
    \end{bmatrix}
    ,
  \end{align*}
  where $\alpha^c$ is the complement of $\alpha$
  and the matrix $\Ueabs$ is defined as
  \begin{align*}
    \Ueabs(x^*) =
    \begin{bmatrix}
      \_U(x^*) \\ [e_i^T \Sigma^* \pd_x z(x^*) \_U(x^*)]_{i \notin \alpha}
    \end{bmatrix}
  \end{align*}
  with $\_U(x)$ spanning $\ker(\Jeabs(x))$.
  Using the special structure of \eqref{eq:be-nlp} and comparing the derivatives of $\_\Lagr(\_y^*, \_\lambda^*)$
  and $\Lagr(y^*, \lambda^*)$, the Hessian becomes:
  \begin{align*}
    \Heabs(\_y^*, \_\lambda^*) =
    \begin{bmatrix}
      I & 0 & 0 & 0 \\
      0 & I & 0 & 0 \\
      0 & 0 & P_{(\alpt)^c} & 0 \\
      0 & 0 & 0 & P_{(\alpw)^c}
    \end{bmatrix}
    \begin{bmatrix}
      \pd_{11} \Lagr & 0 & \pd_{12} \Lagr & 0 \\
      0 & 0 & 0 & 0 \\
      \pd_{21} \Lagr & 0 & \pd_{22} \Lagr & 0 \\
      0 & 0 & 0 & 0 \\
    \end{bmatrix}
    \begin{bmatrix}
      I & 0 & 0 & 0 \\
      0 & I & 0 & 0 \\
      0 & 0 & P_{(\alpt)^c}^T & 0 \\
      0 & 0 & 0 & P_{(\alpw)^c}^T
    \end{bmatrix}
  \end{align*}
  All partial derivatives of $\Lagr$ are evaluated at $(y^*, \lambda^*)$.
\ifcase1
  Moreover, $\Jeabs$ has the form
  \begin{align*}
    \Jeabs(x) =
    \Jeabs(t, w) =
    \begin{bmatrix}
      \pd_t \cE(t,\abs{\zt(t)}) & 0 \\
      \pd_t \cI(t,\abs{\zt(t)}) & -\Sigma^w \\
      [e_i^T\pd_t \zt(t)]_{i\in \alpt} & 0 \\
      0 & [e_i^T I]_{i\in \alpw}
    \end{bmatrix}
    ,
  \end{align*}
\else
  Moreover, $\Jeabs(x) = \Jeabs(t, w)$ has the form
  derived in \cref{lem:e-likq},
\fi
  and thus its nullspace is spanned by
  \begin{align*}
    \_U(x) = \begin{bmatrix} U(t) \\ \Sigma^w \pd_t \cI \end{bmatrix},
  \end{align*}
  where $U(t)$ spans the nullspace of
\ifcase1
  \begin{align*}
    \Jabs(t) =
    \begin{bmatrix}
      \pd_t \cE(t, \abs{\zt(t)}) \\
      \pd_t \cA(t, \abs{\zt(t)}) \\
      [e_i^T\pd_t \zt(t)]_{i \in \alpt}
    \end{bmatrix}
    .
  \end{align*}
\else
  $\Jabs(t)$ from \cref{def:likq}.
\fi
  Using this and $\pd_t \zw(w) = I$, the matrix $\Ueabs$ reads
  \begin{align*}
    \Ueabs(x) =
    \begin{bmatrix}
      U(t) \\ \Sigma^w \pd_t \cI \\
      [e_i^T \Sigma^t \pd_t \zt(t) U(t)]_{i \notin \alpt} \\
      [e_i^T \pd_t \cI]_{i \notin \alpw}
    \end{bmatrix}
    .
  \end{align*}
  Finally, we have
  \begin{align*}
    0 \le
    \Ueabs(x^*)^T \Heabs(\_y^*, \_\lambda^*) \Ueabs(x^*) =
    \Uabs(t^*)^T \Habs(y^*, \lambda^*) \Uabs(t^*)
  \end{align*}
  with $\Uabs(t)$ from \eqref{eq:Uabs} and
  \begin{equation*}
    \Habs(y, \lambda) =
    \begin{bmatrix}
      I & 0 \\
      0 & P_{(\alpt)^c}
    \end{bmatrix}
    \begin{bmatrix}
      \pd_{11} \Lagr(y, \lambda) & \pd_{12} \Lagr(y, \lambda) \\
      \pd_{21} \Lagr(y, \lambda) & \pd_{22} \Lagr(y, \lambda)
    \end{bmatrix}
    \begin{bmatrix}
      I & 0 \\
      0 & P_{(\alpt)^c}^T
    \end{bmatrix}
    .
  \end{equation*}
  This proves the claim.
\end{proof}

\begin{theorem}[Second Order Sufficient Conditions for \eqref{eq:i-anf}]
  Assume that $y^* = (t^*,(\zt)^*)$ is kink stationary for \eqref{eq:i-anf}
  with a Lagrange multiplier vector $\lambda^*$ that satisfies
  strict complementarity for $\lamI^*$ and strict normal growth,
  \begin{equation*}
    [\lamE^T \pd_2 \cE -
    \lamI^T \pd_2 \cI +
    \lamZ^T \pd_2 \cZ]_i > \abs{(\lamZ)_i}, \quad i \in \alpt(t^*).
  \end{equation*}
  Assume further that LIKQ holds at $t^*$, and that
  \begin{equation*}
    \Uabs(t^*)^T \Habs(y^*, \lambda^*) \Uabs(t^*) > 0.
  \end{equation*}
  Then, $(t^*, (\zt)^*)$ is a strict local minimizer of \eqref{eq:i-anf}.
\end{theorem}
\begin{proof}
  As before we consider the slack reformulation
  \eqref{eq:e-anf} of \eqref{eq:i-anf}.
  The assumption of strict complementarity for $\lamI^*$
  and strict normal growth for \eqref{eq:i-anf}
  implies strict normal growth for \eqref{eq:e-anf}.
  Moreover, the previous proof shows that the condition
  \begin{equation*}
    \Uabs(t^*)^T \Habs(y^*, \lambda^*) \Uabs(t^*) > 0
  \end{equation*}
  is equivalent to
  \begin{align*}
    \Ueabs(x^*)^T \Heabs(\_y^*, \_\lambda^*) \Ueabs(x^*) > 0,
  \end{align*}
  which can be reformulated using the variable splitting
  of \cite{Hegerhorst_Steinbach:2019}.
  Then, \cite[Theorem 5.19]{Hegerhorst_Steinbach:2019} can be applied,
  which gives the assertion.
\end{proof}

\begin{theorem}
  Assume that $(t^*, (\zt)^*)$ is kink stationary for \eqref{eq:i-anf}
  with Lagrange multiplier vector $\lambda^*$
  such that strict complementarity and strict normal growth are satisfied.
  Assume further that LIKQ holds at $t^*$.
  Then,
  \begin{align*}
    \Umpcc(y^*)^T\! \Hmpcc(y^*, \lambda^*) \Umpcc(y^*) \ge 0
    \iff
    \Uabs(t^*)^T\! \Habs(t^*, (\zt)^*, \lambda^*) \Uabs(x^*) \ge 0,
  \end{align*}
  where $y^* = (t^*, (\ut)^*, (\vt)^*) = (t^*, [(\zt)^*]^+, [(\zt)^*]^-)$.
  The equivalence holds also with strict inequalities.
\end{theorem}
\begin{proof}
  Using that $u^*+v^*=\abs{(\zt)^*}$ and $(\zt)^*=\zt(t^*)$ the matrix $\Umpcc(y^*)$ in \eqref{eq:Umpcc} reads
  \begin{align*}
    \Umpcc(y^*) =
    \begin{bmatrix}
      I \\ +\PUt+ \pd_t \zt(t^*) \\ 0 \\ -\PVt+ \pd_t \zt(t^*) \\ 0
    \end{bmatrix}
    \tilde{U}(y^*)
  \end{align*}
  with $\tilde{U}(y^*)$ defined in \eqref{eq:U1}.
  The Lagrangians of \eqref{eq:i-mpcc} and \eqref{eq:i-anf}, respectively, are
  \begin{align*}
    \Lc(y, \lambda) &= f(t) +
    \lamE^T \cE(t, \ut + \vt) -
    \lamI^T \cI(t, \ut + \vt) +
    \lamZ^T [\cZ(t, \ut + \vt) - (\ut - \vt)], \\
    \Lagr(t, \abs{\zt}, \lambda) &= f(t) +
    \lamE^T \cE(t, \abs{\zt}) -
    \lamI^T \cI(t, \abs{\zt}) +
    \lamZ^T [\cZ(t, \abs{\zt}) - \setP_{\alpha^c}^T\setP_{\alpha^c} \Sigma^t \abs{\zt}].
  \end{align*}
  Thus $\Umpcc(y^*)^T\Hmpcc(y^*, \lambda^*)\Umpcc(y^*) = \Umpcc(y^*)^T\nabla_{yy} \Lc(y^*, \lambda^*)\Umpcc(y^*)$ can be written
  with $(\zt)^* = (\ut)^* - (\vt)^*$ and using that $\tilde{U}(y^*)=U(t^*)$, as
  \begin{equation*}
    \begin{bmatrix}
    U(y^*) \\ +\PUt+ \pd_t\zt(t^*)U(y^*) \\ -\PVt+ \pd_t\zt(t^*)U(y^*)
    \end{bmatrix}^T
    \begin{bmatrix}
      \hfill H_{11} & \hfill H_{21} \PUt+^T & \hfill H_{21} \PVt+^T \\
      \PUt+ H_{12} & \PUt+ H_{22} \PUt+^T & \PUt+ H_{22} \PVt+^T \\
      \PVt+ H_{12} & \PVt+ H_{22} \PUt+^T & \PVt+ H_{22} \PVt+^T
    \end{bmatrix}
    \begin{bmatrix}
    U(y^*) \\ +\PUt+ \pd_t\zt(t^*)U(y^*) \\ -\PVt+ \pd_t\zt(t^*)U(y^*)
    \end{bmatrix}
  \end{equation*}
  where $H_{ij} \define \pd_i \pd_j \Lagr(t^*, (\zt)^*, \lambda^*)$.
  Now, since $\setUt_+ \cup \setVt_+ = (\alpt)^c$,
  the left-hand inequality of the claim reads
  \begin{align*}
    \def\arraystretch{1.3}
    \begin{bmatrix}
      U(y^*) \\  P_{(\alpt)^c} \Sigma \pd_t \zt(t^*) U(y^*)
    \end{bmatrix}^T
    \begin{bmatrix}
      \hfill H_{11} & \hfill H_{21} P_{(\alpt)^c}^T \\
      P_{(\alpt)^c} H_{12} & P_{(\alpt)^c} H_{22} P_{(\alpt)^c}^T
    \end{bmatrix}
    \begin{bmatrix}
      U(y^*) \\ P_{(\alpt)^c} \Sigma \pd_t \zt(t^*) U(y^*)
    \end{bmatrix}
    \ge 0.
  \end{align*}
  This is $\Uabs(t^*)^T \Habs(t^*, (\zt)^*, \lambda^*) \Uabs(t^*) \ge 0$.
\end{proof}

Note that the previous theorem can be used to transfer
the second order conditions for \eqref{eq:i-anf}
and \eqref{eq:i-mpcc} into each other.
This follows from the equivalence of LIKQ and MPCC-LICQ
by \cref{th:likq-licq-i} and from the equivalence
of stationarity concepts by \cref{thm:s-stat-is-kink-stat}.

%================================================================================
\section{Conclusions and Outlook}\label{sec:conclusion}
%================================================================================

We have shown that general abs-normal NLPs
are essentially the same problem class as MPCCs.
The two problem classes have corresponding constraint qualifications,
stationarity concepts, and optimality conditions of first and second order.
We have also shown that the slack reformulation
from \cite{Hegerhorst_Steinbach:2019},
which is useful to simplify derivations under LIKQ,
does not preserve IDKQ and has other subtle drawbacks
like non-uniqueness of slack variables.
We have not considered counterpart abs-normal NLPs of general MPCCs
as in \cite{Hegerhorst_et_al:2019:MPEC1}.
This would provide a different perspective on the equivalence
of the two problem classes but no additional insight.
It is hoped that the identities revealed may serve to transfer algorithms for the solution of MPCCs to the young field of abs-normal forms and abs-normal NLP. Vice versa, active signature algorithms for abs-normal forms, such as
SALMIN \cite{Griewank_Walther_2018}, may be applicable to MPCCs.
Relations between the two problem classes under weaker
constraint qualifications of Abadie type and Guignard type
are the subject of part two of this research and are put forward in a companion article.

\bibliographystyle{jnsao}
\bibliography{abs_normal_nlp}

\begin{thebibliography}{10}

\bibitem{FlegelDiss}
M{.\nobreak\kern 0.33333em}Flegel, \emph{Constraint Qualifications and
  Stationarity Concepts for Mathematical Programs with Equilibrium
  Constraints}, Dissertation, Universit\"at W\"urzburg, 2005,
  \url{https://opus.bibliothek.uni-wuerzburg.de/frontdoor/index/index/year/2005/docId/1068}.

\bibitem{Griewank_2013}
A{.\nobreak\kern 0.33333em}Griewank, On stable piecewise linearization and
  generalized algorithmic differentiation, \emph{Optimization Methods and
  Software} 28 (2013),  1139--1178,
  \href{https://dx.doi.org/10.1080/10556788.2013.796683}{\nolinkurl{doi:10.1080/10556788.2013.796683}}.

\bibitem{Griewank_Walther_2016}
A{.\nobreak\kern 0.33333em}Griewank and A{.\nobreak\kern 0.33333em}Walther,
  First- and second-order optimality conditions for piecewise smooth objective
  functions, \emph{Optimization Methods and Software} 31 (2016),  904--930,
  \href{https://dx.doi.org/10.1080/10556788.2016.1189549}{\nolinkurl{doi:10.1080/10556788.2016.1189549}}.

\bibitem{Griewank_Walther_2017}
A{.\nobreak\kern 0.33333em}Griewank and A{.\nobreak\kern 0.33333em}Walther,
  Characterizing and testing subdifferential regularity for piecewise smooth
  objective functions, \emph{{SIAM} Journal on Optimization} 29 (2019),
  1473--1501,
  \href{https://dx.doi.org/10.1137/17M115520X}{\nolinkurl{doi:10.1137/17m115520x}}.

\bibitem{Griewank_Walther_2018}
A{.\nobreak\kern 0.33333em}Griewank and A{.\nobreak\kern 0.33333em}Walther,
  Finite convergence of an active signature method to local minima of piecewise
  linear functions, \emph{Optimization Methods and Software} 34 (2019),
  1035--1055,
  \href{https://dx.doi.org/10.1080/10556788.2018.1546856}{\nolinkurl{doi:10.1080/10556788.2018.1546856}}.

\bibitem{Griewank_Walther_2019}
A{.\nobreak\kern 0.33333em}Griewank and A{.\nobreak\kern 0.33333em}Walther,
  Relaxing kink qualifications and proving convergence rates in piecewise
  smooth optimization, \emph{{SIAM} Journal on Optimization} 29 (2019),
  262--289,
  \href{https://dx.doi.org/10.1137/17m1157623}{\nolinkurl{doi:10.1137/17m1157623}}.

\bibitem{Hegerhorst-Schultchen2020}
L.\,C{.\nobreak\kern 0.33333em}Hegerhorst-Schultchen, \emph{Optimality
  Conditions for Abs-Normal {NLP}s}, Dissertation, Leibniz Universit\"at
  Hannover, 2020,
  \href{https://dx.doi.org/10.15488/9867}{\nolinkurl{doi:10.15488/9867}}.

\bibitem{Hegerhorst_et_al:2019:MPEC1}
L.\,C{.\nobreak\kern 0.33333em}Hegerhorst-Schultchen, C{.\nobreak\kern
  0.33333em}Kirches, and M.\,C{.\nobreak\kern 0.33333em}Steinbach, On the
  relation between MPECs and optimization problems in abs-normal form,
  \emph{Optimization Methods and Software} 35 (2020),  565--575,
  \href{https://dx.doi.org/10.1080/10556788.2019.1588268}{\nolinkurl{doi:10.1080/10556788.2019.1588268}}.

\bibitem{Hegerhorst_Steinbach:2019}
L.\,C{.\nobreak\kern 0.33333em}Hegerhorst-Schultchen and M.\,C{.\nobreak\kern
  0.33333em}Steinbach, On first and second order optimality conditions for
  abs-normal NLP, \emph{Optimization} 69 (2020),  2629--2656,
  \href{https://dx.doi.org/10.1080/02331934.2019.1626386}{\nolinkurl{doi:10.1080/02331934.2019.1626386}}.

\bibitem{Luo_Pang_Ralph_1996}
Z{.\nobreak\kern 0.33333em}Luo, J{.\nobreak\kern 0.33333em}Pang, and
  D{.\nobreak\kern 0.33333em}Ralph, \emph{Mathematical Programs with
  Equilibrium Constraints}, Cambridge University Press, 1996.

\bibitem{Scheel_Scholtes_2000}
H{.\nobreak\kern 0.33333em}Scheel and S{.\nobreak\kern 0.33333em}Scholtes,
  Mathematical programs with complementarity constraints: stationarity,
  optimality, and sensitivity, \emph{Mathematics of Operations Research} 25
  (2000),  1--22,
  \href{https://dx.doi.org/10.1287/moor.25.1.1.15213}{\nolinkurl{doi:10.1287/moor.25.1.1.15213}}.

\bibitem{Ye:2005}
J{.\nobreak\kern 0.33333em}Ye, Necessary and sufficient optimality conditions
  for mathematical programs with equilibrium constraints, \emph{Journal of
  Mathematical Analysis and Applications} 307 (2005),  350--369,
  \href{https://dx.doi.org/10.1016/j.jmaa.2004.10.032}{\nolinkurl{doi:10.1016/j.jmaa.2004.10.032}}.

\end{thebibliography}

\end{document}